\documentclass[12pt,longbibliography]{article}

\usepackage[utf8]{inputenc} 
\usepackage[T1]{fontenc}

\usepackage{graphicx}
\usepackage{amsmath}
\usepackage{amsfonts}
\usepackage{mathrsfs}           
\usepackage{amsthm}
\usepackage{amssymb}
\usepackage{color}
\usepackage{enumerate}
\usepackage{comment}
\usepackage{bbold} 

\usepackage[a4paper, margin=2.7cm]{geometry}

\usepackage[colorlinks=true,linkcolor=blue,citecolor=blue,urlcolor=blue,breaklinks]{hyperref}

\usepackage{bbm}

\usepackage{breakurl}
\usepackage{url}

\def\C{\mathbb{C}}
\def\N{\mathbb{N}}
\def\R{\mathbb{R}}

\def\d{\,\mathrm{d}}

\newtheorem{thm}{Theorem}[section]
\newtheorem{cor}[thm]{Corollary}
\newtheorem{lem}[thm]{Lemma}
\newtheorem{prp}[thm]{Proposition}

\theoremstyle{definition}
\newtheorem{dfn}[thm]{Definition}
\theoremstyle{remark}
\newtheorem{rem}[thm]{Remark}
\theoremstyle{example}
\newtheorem{ex}[thm]{Example}

\title{Ultracontractivity of heat semigroups with non-local Robin
boundary conditions via two-sided bounds for a positive eigenfunction}

\author{Christoph Schwerdt$^1$}
\date{
	$^1$ Institute of Mathematics, University of Rostock,\\
	Ulmenstra\ss e 69, 18 057 Rostock, Germany \\
	\ \\
	\today
}

\begin{document}

\maketitle

\begin{abstract}
We study heat semigroups on bounded Lipschitz domains
$\Omega \subset \R^{d}$ with dimension $d>2$ under non-local Robin boundary
conditions. The boundary operator $B \in \mathcal{L}( \mathrm{L}^{2}(\partial\Omega))$ 
is allowed to destroy the positivity of the semigroup. We assume that there 
exists a positive operator $C \in \mathcal{L}(\mathrm{L}^{2}(\partial\Omega))$ such that
$$
|Bu| \leq C|u|
\quad\text{for every }u\in \mathrm{L}^{2}(\partial\Omega),
\qquad
C{\bf 1}\in \mathrm{L}^{\infty}(\partial\Omega).
$$
Under this assumption, we prove that the semigroup generated by the
corresponding uniformly elliptic operator is ultracontractive. More
precisely, its norm from $\mathrm{L}^{2}(\Omega)$ to $\mathrm{L}^{\infty}(\Omega)$ has
short-time order $t^{-d/4}$.\\
\ \\
The main step is the construction of a positive eigenfunction $\phi$
of a dominating elliptic operator such that
$$
0 \ < \ \delta \ \leq \ \phi \ \leq \ M
$$
almost everywhere in $\Omega$. The upper bound is obtained by power
truncations and an iteration of Sobolev exponents. The lower bound
follows by comparison with the Neumann semigroup. The two-sided
estimate gives an $\mathrm{L}^{\infty}$-bound for the comparison semigroup.
Nash's inequality and duality then yield ultracontractivity. 
\end{abstract}

\medskip


\tableofcontents

\section{Introduction}

Let $\Omega \subset\R^{d}$ be a bounded Lipschitz domain and let
$A \colon \Omega \to \R^{d\times d}$ be a bounded measurable and
uniformly elliptic coefficient matrix. We consider the heat equation
$$
\frac{\partial u}{\partial t}-\operatorname{div}(A\nabla u) \ = \ 0
$$
in $(0,\infty) \times \Omega$ with a non-local Robin boundary condition 
formally given by
$$
\nu \cdot A\nabla u + Bu \ = \ 0
$$
on $(0,\infty) \times \partial\Omega$. Here, $\nu$ denotes the outer unit 
normal on $\partial\Omega$ and
$$
B \in \mathcal{L}\left( L^{2}(\partial\Omega) \right)
$$
is a bounded boundary operator. The corresponding elliptic operator
$L(A,B)$ is defined by a sesquilinear form on $H^{1}(\Omega)$ and
$-L(A,B)$ generates a $C_{0}$-semigroup
$\left( \mathrm{e}^{-tL(A,B)} \right)_{t\geq0}$
in $L^{2}(\Omega)$.\\
\ \\
If $B$ is a multiplication operator, one obtains a classical local
Robin boundary condition. General operators $B$, however, can describe
non-local interactions on the boundary. They may also destroy the
positivity of the associated semigroup. This makes standard arguments
for positive heat semigroups unavailable.\\
\ \\
We study the ultracontractivity of $\left( \mathrm{e}^{-tL(A,B)} \right)_{t\geq0}$. 
This means that
$$
\mathrm{e}^{-tL(A,B)} \in \mathcal{L} \left( \mathrm{L}^{2}(\Omega), \mathrm{L}^{\infty}(\Omega) \right)
$$
for every $t>0$. Thus the semigroup has a regularising effect.
Jochen Glück and Jonathan Mui proved an ultracontractivity result for this problem in
\cite[Theorem 3.2]{Glueck_Mui_2026}. They assume that
$B \in \mathcal{L}( \mathrm{L}^{2}(\partial\Omega) )$ acts boundedly on both
$\mathrm{L}^{1}(\partial\Omega)$ and $\mathrm{L}^{\infty}(\partial\Omega)$. Under this
assumption, they obtain ultracontractivity of the semigroup and of its adjoint.\\
\ \\
Our aim is to weaken the assumption on the boundary operator. We assume that there exists a
positive operator
$C \in \mathcal{L}\left( \mathrm{L}^{2}(\partial\Omega) \right)$ such that
$$
|Bu| \ \leq \ C|u|
$$
for every $u \in \mathrm{L}^{2}(\partial\Omega)$ and
$C {\bf1} \in \mathrm{L}^{\infty}(\partial\Omega)$.
This assumption implies that $B$ acts boundedly on
$\mathrm{L}^{\infty}(\partial\Omega)$. However, it does not imply that $B$ acts boundedly
on $\mathrm{L}^{1}(\partial\Omega)$. Hence our main theorem applies to a
strictly larger class of boundary operators than the ultracontractivity result of Glück and Mui. 
A simple rank-one operator showing that this difference is genuine is presented in
Section \ref{comparison_Glueck_Mui}.\\
\ \\
For $d>2$, our main result gives constants $K>0$ and $\lambda>0$ such
that
$$
\left\| \mathrm{e}^{-tL(A,B)}u \right\|_{\mathrm{L}^{\infty}(\Omega)} \ \leq \
Kt^{-d/4} \exp\left( t \left(\lambda+\frac{\alpha}{2} \right) \right)
\|u\|_{\mathrm{L}^{2}(\Omega)}
$$
for every $t>0$ and every $u \in \mathrm{L}^{2}(\Omega)$. Thus we obtain the same
short-time order $t^{-d/4}$ as in the result of Glück and Mui. Our
weaker assumption does not automatically give the corresponding
estimate for the adjoint semigroup. Such an estimate follows if an
analogous assumption is imposed on $B^{\ast}$.\\
\ \\
The main idea of the proof is the construction of a suitable positive
comparison semigroup. The domination condition on $B$ and $C$ gives
$$
\left| \mathrm{e}^{-tL(A,B)}u \right| \ \leq \  \mathrm{e}^{-tL(A,-C)}|u|.
$$
We then construct a positive eigenfunction $\phi$ of $L(A,-C)$ and
prove that there exist constants $\delta,M>0$ such that
$$
0 \ < \ \delta \ \leq \ \phi \ \leq \ M
$$
almost everywhere in $\Omega$. These two estimates are the central
part of our argument.\\
\ \\
The existence of $\phi$ follows from the Krein--Rutman theorem applied
to a positive compact resolvent of $L(A,-C)$. The upper bound is proved
by power truncations and an iteration of Sobolev exponents. For the
lower bound, we compare the positive semigroup generated by
$-L(A,-C)$ with the Neumann semigroup. We then use the convergence of
the Neumann semigroup to the mean-value projection in 
$\mathrm{L}^{2}(\Omega)$.\\
\ \\
The two-sided estimate for $\phi$ gives an 
$\mathrm{L}^{\infty}(\Omega)$-bound
for the comparison semigroup and therefore also for
$e^{-tL(A,B)}$. By duality, the adjoint semigroup acts boundedly on
$\mathrm{L}^{1}(\Omega)$. We apply Nash's inequality to the adjoint semigroup
and obtain an $\mathrm{L}^{1}(\Omega)$-to-$\mathrm{L}^{2}(\Omega)$ estimate. 
A second duality argument then gives the asserted
$\mathrm{L}^{2}(\Omega)$-to-$\mathrm{L}^{\infty}(\Omega)$ estimate.\\
\ \\
The article is organised as follows. In Section \ref{Preparations} we introduce the
forms and operators used throughout the article and prove the required
semigroup domination. Section \ref{section_main_theorems} contains the main theorem. 
In Section \ref{comparison_Glueck_Mui} we compare our assumptions and conclusions 
with those of Jochen Glück and Jonathan Mui. Section \ref{proof_main_theorem} proves the 
main theorem, assuming the existence of the eigenfunction with the two-sided estimate stated
above. Finally, Section \ref{bounded_Eigenfunction} constructs this eigenfunction. We first prove
its existence and upper bound and then establish its strictly positive lower bound.

\section{Preliminaries}\label{Preparations}

\subsection*{Setting and notation}

Throughout this article, all function spaces are considered over the
complex field unless explicitly stated otherwise. Let $\Omega \subset \R^{d}$ for $2 \leq d \in \N$ 
be a bounded Lipschitz domain. As usual, a domain is understood to be a non-empty, open and connected set.\\
\ \\
Inequalities between functions are understood pointwise almost
everywhere. In particular, $u \leq v$ means that $u$ and $v$ are
real-valued and that $u(x)\leq v(x)$ for almost every $x \in \Omega$.
A function $u$ is called positive if $u \geq 0$ almost everywhere in $\Omega$.
It is called strictly positive if $u > 0$ almost everywhere.\\
\ \\
A linear operator $T$ on an $\mathrm{L}^{p}$-space is called positive if
$Tu \geq 0$ for every $u \geq 0$. We denote the constant function with
value $1$ by ${\bf1}$. Its underlying domain will always be
clear from the context.\\
\ \\
We write $\mathcal{L}(X,Y)$ for the space of bounded linear operators
from $X$ to $Y$ and 
$$
\mathcal{L}(X)=\mathcal{L}(X,X). 
$$
Furthermore,
$$
\gamma \ \colon \ H^1(\Omega) \ \to \ \mathrm{L}^{2}(\partial\Omega)
$$
denotes the trace operator.\\
\ \\
Let $A \colon \Omega \to \R^{d \times d}$ be
a matrix-valued function whose coefficients $a_{ij} \colon \Omega \to \R$ are bounded and measurable functions. 
We further assume that $A$ is uniformly elliptic on $\Omega$, that is
$$
\exists \alpha > 0 \ \forall \xi \in \C^{d} \ : \ \Re \left(  \left( A(x) \xi \right)^{T} \overline{\xi} \right) 
\geq \alpha \left| \xi \right|^{2}. 
$$
for almost every $x \in \Omega$. The boundary condition is described by  
$B \in \mathcal{L} \left( \mathrm{L}^{2}\left( \partial \Omega \right) \right)$.\\

\begin{lem}\label{gamma}
For every $\varepsilon > 0$ there exists a constant $\beta(\varepsilon) > 0$ such that
$$
\| \gamma(u) \|_{\mathrm{L}^{2}(\partial \Omega)}^{2} \ \leq \ \varepsilon \| \nabla u \|_{\mathrm{L}^{2}(\Omega)^{d}}^{2}
+ \beta(\varepsilon) \| u \|_{\mathrm{L}^{2}(\Omega)}^{2}
$$
for every $u \in H^{1}(\Omega)$ where $\beta(\varepsilon) = \mathcal{O}(\varepsilon^{-1})$ for $\varepsilon \to 0^{+}$. 
Furthermore,  there exists a constant $k > 0$ independent of $u$ and $\varepsilon$ such that
$$
\| \gamma(u) \|_{\mathrm{L}^{2}(\partial \Omega)}^{2} \ \leq \ \varepsilon \| \nabla u \|_{\mathrm{L}^{2}(\Omega)^{d}}^{2}
+ k(1+ \varepsilon^{-1}) \| u \|_{\mathrm{L}^{2}(\Omega)}^{2}
$$
is true for every $\varepsilon > 0$ and any $u \in H^{1}(\Omega)$.
\end{lem} 

\begin{proof} 
The first assertion follows from \cite[Lemma 2.5]{Gesztesy_Mitrea_2008}. Since 
$\beta(\varepsilon)=O(\varepsilon^{-1})$ as
$\varepsilon \to 0^{+}$, there exist constants $c>0$ and $\varepsilon_{0}>0$ such that
$$
\beta(\varepsilon) \leq \frac{c}{\varepsilon}
$$ 
for every $0 < \varepsilon \leq \varepsilon_{0}$. Hence
$$
\| \gamma(u) \|_{\mathrm{L}^{2}(\partial \Omega)}^{2} \ \leq \ \varepsilon \| \nabla u \|_{\mathrm{L}^{2}(\Omega)^{d}}^{2}
+ c\varepsilon^{-1} \| u \|_{\mathrm{L}^{2}(\Omega)}^{2}
$$
for every $0 < \varepsilon \leq \varepsilon_{0}$. If $\varepsilon > \varepsilon_{0}$, then 
\begin{align*}
\| \gamma(u) \|_{\mathrm{L}^{2}(\partial \Omega)}^{2} & \leq \varepsilon_{0} \| \nabla u \|_{\mathrm{L}^{2}(\Omega)^{d}}^{2}
+ \beta(\varepsilon_{0}) \| u \|_{\mathrm{L}^{2}(\Omega)}^{2} \\
& < \varepsilon \| \nabla u \|_{\mathrm{L}^{2}(\Omega)^{d}}^{2}
+ \beta(\varepsilon_{0}) \| u \|_{\mathrm{L}^{2}(\Omega)}^{2}.
\end{align*}
Thus the second assertion holds for every $k \geq \beta(\varepsilon_{0}) + c$.\\
\end{proof}

\begin{cor}\label{Operator_B}
For every $\varepsilon > 0$ there exists a constant $\beta(\varepsilon) > 0$ such that
\begin{align*}
& \Re \int_{\Omega} A \nabla u \cdot \overline{\nabla u} \d x + 
\Re \int_{\partial \Omega} B\gamma(u) \overline{\gamma(u)} \d \sigma(x) & \\
& \geq \ \left( \alpha - \varepsilon \| B \|_{\mathrm{L}^{2} \to \mathrm{L}^{2}} \right)  \| \nabla u \|_{\mathrm{L}^{2}(\Omega)^{d}}^{2}
- \beta(\varepsilon) \| B \|_{\mathrm{L}^{2} \to \mathrm{L}^{2}}   \| u \|_{\mathrm{L}^{2}(\Omega)}^{2} &
\end{align*}
is true for every $u \in H^{1}(\Omega)$.
\end{cor}

\begin{proof}
Let $\varepsilon > 0$. By Lemma \ref{gamma} we have
\begin{align*}
- \Re \int_{\partial \Omega} B\gamma(u) \overline{\gamma(u)} \d \sigma(x) & 
\leq \left| \int_{\partial \Omega} B\gamma(u) \overline{\gamma(u)} \d \sigma(x) \right| \leq 
\| B \|_{\mathrm{L}^{2} \to \mathrm{L}^{2}} \| \gamma(u) \|_{\mathrm{L}^{2}(\partial \Omega)}^{2} & \\
\ \\
& \leq \varepsilon \| B \|_{\mathrm{L}^{2} \to \mathrm{L}^{2}}  \| \nabla u \|_{\mathrm{L}^{2}(\Omega)^{d}}^{2}
+ \beta(\varepsilon) \| B \|_{\mathrm{L}^{2} \to \mathrm{L}^{2}}   \| u \|_{\mathrm{L}^{2}(\Omega)}^{2} &
\end{align*}
for every $u \in H^{1}(\Omega)$. On the other hand, the uniform ellipticity gives
$$
\Re \int_{\Omega} A \nabla u \cdot \overline{\nabla u} \d x \ \geq \ \alpha \| \nabla u \|_{\mathrm{L}^{2}(\Omega)^{d}}^{2}.
$$
Combining these two estimates proves the assertion.\\
\end{proof}

\subsection{The $C_{0}$-semigroup $\mathrm{e}^{-tL(A,B)}$ in $\mathrm{L}^{2}\left( \Omega \right)$
for $B \in \mathcal{L}\left( \mathrm{L}^{2}\left( \partial \Omega \right) \right)$} \label{form_a}

We set 
$$
\varepsilon_{0} = \frac{\alpha}{2} (\| B \|_{\mathrm{L}^{2} \to \mathrm{L}^{2}} +1)^{-1}.
$$ 
There exists a constant $\beta(\varepsilon_{0}) > 0$ in the sense of Corollary \ref{Operator_B}.
Choose $\lambda_{0}$ sufficiently large that 
\begin{equation}
\lambda_{0} \geq \beta(\varepsilon_{0}) \| B \|_{\mathrm{L}^{2} \to \mathrm{L}^{2}}.
\end{equation}
We now define the corresponding sesquilinear forms in  $\mathrm{L}^{2}\left( \Omega \right)$.\\

\begin{dfn}
Define an auxiliary sesquilinear form $\tilde{a}(A, B)$ in $\mathrm{L}^{2}\left( \Omega \right)$ by
$$
\tilde{a}(A, B)(u,v) = \int_{\Omega} A \nabla u \cdot \overline{\nabla v} \d x + 
\int_{\partial \Omega} B \gamma(u) \ \overline{\gamma(v)} \d \sigma(x) + \lambda_{0} \int_{\Omega} u \ \overline{v} \d x
$$
for $u,v \in D \left( \tilde{a}(A, B) \right) = H^{1}\left( \Omega \right)$ and the form $a(A, B)$ by
\begin{equation}\label{form_a_B}
a(A, B)(u,v) = \int_{\Omega} A \nabla u \cdot \overline{\nabla v} \d x + 
\int_{\partial \Omega} B \gamma(u) \ \overline{\gamma(v)} \d \sigma(x)
\end{equation}
for $u,v \in D \left( a(A, B) \right) = D \left( \tilde{a}(A, B) \right)$.\\
\end{dfn}

\begin{rem} \label{a(B)_Nash}
The form $a(A,B)$ need not be accretive. In contrast,
Corollary \ref{Operator_B} and the choice of $\lambda_{0}$ imply that the
auxiliary form $\widetilde a(A,B)$ is accretive, since
\begin{equation}\label{tilde_a_accretive}
\Re \tilde{a}(A, B)(u,u) \geq  \frac{\alpha}{2} \| \nabla u \|_{\mathrm{L}^{2}(\Omega)^{d}}^{2} \geq 0
\end{equation}
for every $u \in D \left( \tilde{a}(A, B) \right)$. \\
\end{rem}

Let us show the continuity and closedness of the auxiliary form $\tilde{a}(A, B)$ next. We define the induced norm 
of $\tilde{a}(A, B)$ by
$$
\| u \|_{\tilde{a}(A, B)} = \left( \Re \tilde{a}(A, B)(u,u) +  \| u \|_{\mathrm{L}^{2}\left( \Omega \right)}^{2} \right)^{\frac{1}{2}}
$$
for $u \in D \left( \tilde{a}(A, B) \right)$. It follows from (\ref{tilde_a_accretive}) that
\begin{equation}\label{H1_inequ}
\| u \|_{\tilde{a}(A, B)} \ = \ \left( \Re \tilde{a}(A, B)(u,u) +  \| u \|_{\mathrm{L}^{2}\left( \Omega \right)}^{2} \right)^{\frac{1}{2}}
\ \geq \ \min \left( \frac{\alpha}{2}, 1 \right)^{1/2} \| u \|_{H^{1}(\Omega)}.\\
\end{equation}

\begin{lem}\label{form_properties}
The form $\tilde{a}(A, B)$ is continuous and closed.
\end{lem}

\begin{proof} \
\begin{enumerate}[i.)]
\item We first prove that $\tilde{a}(A, B)$ is continuous. Note that 
$\| u \|_{\mathrm{L}^{2}\left( \Omega \right)} \leq \| u \|_{\tilde{a}(A, B)}$
is true for every $u \in D \left( \tilde{a}(A, B) \right)$. For $u,v \in D \left( \tilde{a}(A, B) \right)$ we obtain
$$
\left| \tilde{a}(A, B)(u,v) \right| \leq \left| \langle A \nabla u, \nabla v \rangle_{\mathrm{L}^{2}\left( \Omega \right)^{d}}  \right|
+ \left| \langle B \gamma(u), \gamma(v) \rangle_{\mathrm{L}^{2}\left( \partial \Omega \right)} \right| 
+ \lambda_{0} \left| \langle u, v \rangle \right|. 
$$
For the principal part, we have
\begin{align*}
\left| \langle A \nabla u, \nabla v \rangle_{\mathrm{L}^{2}\left( \Omega \right)^{d}}  \right| 
& = \left| \int_{\Omega} \left( A(x) (\nabla u)(x) \right)^{T}  \overline{(\nabla v)(x)} \d x  \right| & \\
& \leq \sum_{l=1}^{d} \sum_{k=1}^{d} \int_{\Omega} \left| a_{lk}(x) \right| \ \left| (\partial_{k}u)(x) \right| 
\ \left| (\partial_{l}v)(x) \right| \ \d x & \\
& \leq \max_{l,k} \left\| a_{l,k} \right\|_{\mathrm{L}^{\infty}\left( \Omega \right)}  \sum_{l=1}^{d} \sum_{k=1}^{d}
\left\| u \right\|_{H^{1}\left( \Omega \right)} \left\| v \right\|_{H^{1}\left( \Omega \right)}  & \\
& \leq d^{2}  \left\| A \right\|_{\mathrm{L}^{\infty}\left( \Omega \right)^{d \times d}} 
\left\| u \right\|_{H^{1}\left( \Omega \right)} \left\| v \right\|_{H^{1}\left( \Omega \right)}.  & 
\end{align*}
Consequently,
\begin{align*}
\left| \tilde{a}(A, B)(u,v) \right| & \leq \left| \langle A \nabla u, \nabla v \rangle_{\mathrm{L}^{2}\left( \Omega \right)^{d}}  \right|
+ \left| \langle B \gamma(u), \gamma(v) \rangle_{\mathrm{L}^{2}\left( \partial \Omega \right)} \right| 
+ \lambda_{0} \left| \langle u, v \rangle \right| & \\
\ \\
& \leq d^{2} \left\| A \right\|_{\mathrm{L}^{\infty}\left( \Omega \right)^{d \times d}}  
\left\| u \right\|_{H^{1}\left( \Omega \right)} \left\| v \right\|_{H^{1}\left( \Omega \right)} 
+ \| B \|_{\mathrm{L}^{2} \to \mathrm{L}^{2}} \ 
\| \gamma (u) \|_{\mathrm{L}^{2}\left( \partial \Omega \right)} \ 
\| \gamma (v) \|_{\mathrm{L}^{2}\left( \partial \Omega \right)} & \\
& + \lambda_{0}  \| u \|_{\mathrm{L}^{2}\left( \Omega \right)} \| v \|_{\mathrm{L}^{2}\left( \Omega \right)} & \\
\ \\
& \leq \left( \ d^{2} \ \| A \|_{\mathrm{L}^{\infty}(\Omega)^{d \times d}}  +   \| B \|_{\mathrm{L}^{2} \to \mathrm{L}^{2}} \ 
\| \gamma \|_{H^{1} \to \mathrm{L}^{2}}^{2} \ \right) 
\| u \|_{H^{1}\left( \Omega \right)} \| v \|_{H^{1}\left( \Omega \right)} & \\
& + \lambda_{0}  \| u \|_{\mathrm{L}^{2}\left( \Omega \right)} \| v \|_{\mathrm{L}^{2}\left( \Omega \right)} & \\
\ \\
& \leq K \| u \|_{\tilde{a}(A, B)} \| v \|_{\tilde{a}(A, B)} & 
\end{align*}
for a sufficiently large constant $K > 0$ independent of $u$ and $v$.
\item We now prove that $\tilde{a}(A, B)$ is closed. Let $(u_{n}) \subset D \left( \tilde{a}(A, B) \right)$ be a  
Cauchy sequence with respect to  $\| \cdot \|_{\tilde{a}(A, B)}$. 
By $( \ref{H1_inequ} )$, $(u_{n})$ is Cauchy in $H^{1}(\Omega)$. Hence there exists 
$u \in H^{1}(\Omega) = D \left( \tilde{a}(A, B) \right)$ such that $u_{n} \to u$ in $H^{1}(\Omega)$.
The estimate in terms of the $H^{1}(\Omega)$-norm obtained in part i)
shows that
$$
\begin{aligned}
\| u_{n} -u \|_{\tilde{a}(A,B)}^{2} &\leq | \tilde{a}(A,B)(u_{n}-u,u_{n} -u) | +\|u_n-u\|_{L^{2}(\Omega)}^{2} \\
& \leq K^{\prime} \|u_{n} - u\|_{H^{1}(\Omega)}^{2} \to 0.
\end{aligned}
$$
Thus $\tilde{a}(A,B)$ is closed.
\end{enumerate}
\end{proof}

\begin{cor}\label{C^1_core}
The subspace $C^{1} \left( \overline{\Omega} \right)$ is dense in $D \left( \tilde{a}(A, B) \right)$ 
with respect to $\| \cdot \|_{\tilde{a}(A, B)}$. 
\end{cor}

\begin{proof}
The space $C^{1}(\overline\Omega)$ is dense in $H^{1}(\Omega)$. Moreover, the estimate
$$
\left| \tilde{a}(A, B)(u,v) \right|  \leq 
K \| u \|_{H^{1}\left( \Omega \right)} \| v \|_{H^{1}\left( \Omega \right)} 
+ \lambda_{0}  \| u \|_{\mathrm{L}^{2}\left( \Omega \right)} \| v \|_{\mathrm{L}^{2}\left( \Omega \right)}
$$
shows that convergence in $H^{1}(\Omega)$ implies convergence with
respect to $\|\cdot\|_{\tilde{a}(A,B)}$.\\ 
This proves the assertion.\\
\end{proof}

\begin{dfn}
We define the operator associated with the form $\tilde{a}(A,B)$ as follows.
\begin{align*}
& D ( \tilde{L}(A, B) ) = & \\
& \left\{ u \in D \left( \tilde{a}(A, B) \right) \ | \ \exists v \in \mathrm{L}^{2}\left( \Omega \right) : 
\tilde{a}(A, B) (u, \varphi) = \langle v, \varphi \rangle_{\mathrm{L}^{2}\left( \Omega \right)} \ \text{ for every } 
\varphi \in D \left( \tilde{a}(A, B) \right) \right\} & 
\end{align*}
as the domain of 
$\tilde{L}(A, B) \colon D ( \tilde{L}(A, B) ) \subset \mathrm{L}^{2}\left( \Omega \right) \to \mathrm{L}^{2}\left( \Omega \right)$ 
which is characterized by the equation
$$
\tilde{a}(A, B)(u, \varphi) =  \langle \tilde{L}(A, B)u, \varphi \rangle_{\mathrm{L}^{2}\left( \Omega \right)}
$$
for $u \in D ( \tilde{L}(A, B) )$ and every $\varphi \in D \left( \tilde{a}(A, B) \right)$. 
\end{dfn}

Since $\tilde{a}(A,B)$ is densely defined, accretive,
continuous and closed, the standard form method implies that
$-\tilde{L}(A,B)$ generates a $C_{0}$-semigroup of contractions
$( \mathrm{e}^{-t\tilde{L}(A, B)} )_{t \geq 0}$ on $\mathrm{L}^{2}\left( \Omega \right)$. 
Let us focus on $a(A, B)$ next.

\begin{prp}\label{Semigroup_adjustment}
Define the operator $L(A, B) = \tilde{L}(A, B) - \lambda_{0}$ in $\mathrm{L}^{2}\left( \Omega \right)$ on the domain 
$D \left( L(A, B) \right) = D ( \tilde{L}(A, B) )$. Then $L(A, B)$ satisfies the following properties:
\begin{enumerate}[i.)]
\item $L(A, B)$ is associated with the form $a(A, B)$, i.e.
$$
a(A, B)(u, \varphi) =  \langle L(A, B)u, \varphi \rangle_{\mathrm{L}^{2}\left( \Omega \right)}
$$
holds for every $u \in D ( L(A, B) )$ and every $\varphi \in D \left( a(A, B) \right)$.
\item $-L(A, B)$ generates the $C_{0}$-semigroup 
$\mathrm{e}^{t\lambda_{0}}\mathrm{e}^{-t\tilde{L}(A, B)}$
on $\mathrm{L}^{2}\left( \Omega \right)$. We write $\mathrm{e}^{-tL(A, B)}$ for 
$\mathrm{e}^{t\lambda_{0}}\mathrm{e}^{-t\tilde{L}(A, B)}$ at every time $t \geq 0$.
\end{enumerate}
\end{prp} 

Both assertions follow directly from the definitions, and we omit the proof.\\

\subsection{Domination of $\mathrm{e}^{-tL(A, B)}$}

Assume in addition that there exists a positive operator $C \in \mathcal{L} \left( \mathrm{L}^{2}( \partial \Omega ) \right)$
such that
\begin{equation}
\left| B u \right| \ \leq \ C \left| u \right|
\end{equation}
for every $u \in \mathrm{L}^{2}( \partial \Omega )$. Using the notation introduced in
Subsection \ref{form_a}, we denote by 
$\tilde{a}(A, -C)$ and $a(A, -C)$ the forms obtained by replacing
$B$ with $-C$. Similarly, we denote the associated operators by $\tilde{L}(A, -C)$ and $L(A, -C)$ 
and their semigroups by $( \mathrm{e}^{-t\tilde{L}(A, -C)} )_{t \geq 0}$ and $( \mathrm{e}^{-tL(A, -C)} )_{t \geq 0}$.\\

\begin{rem}\label{choice_lambda}
We denote the shift $\lambda_{0}$ chosen in Subsection \ref{form_a} for
$\tilde{a}(A,B)$ by $\lambda_{B}$. Similarly, we choose
$\lambda_{C}>0$ sufficiently large for $\tilde{a}(A,-C)$.
In the following, both auxiliary forms and their associated operators
are understood with the common shift
$$
\lambda=\lambda_{B}+\lambda_{C}.
$$
Since $\lambda\geq\lambda_{B}$ and $\lambda\geq\lambda_{C}$,
both auxiliary forms are accretive with this choice.\\
\end{rem}

\begin{lem}\label{inequ_semigroup_B}
The auxiliary operator semigroup $\mathrm{e}^{-t\tilde{L}(A, B)}$ is dominated by the auxiliary operator semigroup
 $\mathrm{e}^{-t\tilde{L}(A,-C)}$, i.e.
$$
\left| \mathrm{e}^{-t\tilde{L}(A, B)}u \right| \ \leq \ \mathrm{e}^{-t\tilde{L}(A,-C)} \left| u \right|
$$ 
holds pointwise almost everywhere in $\Omega$ for every $u \in   \mathrm{L}^{2}( \Omega )$. 
Furthermore 
$$
\left| \mathrm{e}^{-tL(A, B)}u \right| \ \leq \ \mathrm{e}^{-tL(A,-C)} \left| u \right|
$$
is a direct implication where $\lambda$ is chosen as described in Remark \ref{choice_lambda}.
\end{lem}

\begin{proof}
We apply Ouhabaz's domination criterion \cite[Theorem 2.21, p. 60]{Ouhabaz05}. 
We first show that the semigroup $( \mathrm{e}^{-t\tilde{L}(A,-C)} )_{t \geq 0}$ is positive.\\
\begin{enumerate}[i)]
\item Let $u \in C^{1}( \overline{\Omega} )$. We identify $\gamma(u)$ with the restriction of $u$ to $\partial \Omega$.
Since $\nabla(\Re u)^{+}$ and $\nabla(\Re u)^{-}$ have disjoint supports, we have 
\begin{align*} 
& \tilde{a}(A,-C) \left(  (\Re u)^{+},  (\Re u)^{-} \right) = - \int_{\partial \Omega} C (\Re u)^{+} \ (\Re u)^{-} \d \sigma(x)
& \\
& = - \int_{\partial \Omega} \underbrace{C (\Re u)^{+}}_{\geq 0} \ \underbrace{(\Re u)^{-}}_{\geq 0} \d \sigma(x) \leq 0. &
\end{align*}
It follows from \cite[Theorem 2.6, p. 50]{Ouhabaz05} that 
$\mathrm{e}^{-t\tilde{L}(A,-C)}$ is positive for every $t \geq 0$
since $C^{1}( \overline{\Omega} )$ is a core of $\tilde{a}(A,-C)$. 

\item Notice that $H^{1}(\Omega)$ is an ideal in itself by Proposition 2.20 on page 59 in \cite{Ouhabaz05} since
$\mathrm{e}^{-t\tilde{L}(A,-C)}$ are positive. Next let $u,v \in H^{1}(\Omega)$ be such that $u\overline{v} \geq 0$. 
We claim that
$$
\tilde{a}(A,-C)(|u|, |v|) \ \leq \ \Re \tilde{a}(A, B)(u, v).
$$
Since the trace operator commutes with the modulus, we have
$$
\gamma(|u|)=|\gamma(u)| \quad\text{and}\quad \gamma(|v|)=|\gamma(v)|.
$$
For the boundary terms, we obtain
\begin{align*}
- \Re \langle B\gamma(u), \gamma(v) \rangle_{\mathrm{L}^{2}( \partial \Omega )} & 
\leq \left| \langle B\gamma(u), \gamma(v) \rangle_{\mathrm{L}^{2}( \partial \Omega )} \right| 
\leq \langle \left| B\gamma(u) \right|, |\gamma(v)| \rangle_{\mathrm{L}^{2}( \partial \Omega )} & \\
\ \\
& \leq \langle C \left| \gamma(u) \right|, |\gamma(v)| \rangle_{\mathrm{L}^{2}( \partial \Omega )} & 
\end{align*}
which implies
$\langle -C \left| \gamma(u) \right|, \left| \gamma(v) \right|  \rangle_{\mathrm{L}^{2}( \partial \Omega )} 
\leq \Re \langle B \gamma(u), \gamma(v) \rangle_{\mathrm{L}^{2}( \partial \Omega )}$. \\
It remains to compare the remaining parts of the forms. Note that $\mathrm{e}^{-t\tilde{L}(A, 0)}$ are positive operators
which implies
$$
\left| \mathrm{e}^{-t\tilde{L}(A, 0)}w \right| \leq \mathrm{e}^{-t\tilde{L}(A, 0)} \left| w \right|
$$
for every $w \in \mathrm{L}^{2}(\Omega)$. Theorem 2.21 on page 60 in \cite{Ouhabaz05} yields
$$
\tilde{a}(A, 0)(|u|, |v|) \ \leq \ \Re \tilde{a}(A, 0)(u, v).
$$
Combining these estimates, we obtain
$$
\tilde{a}(A, -C)(|u|, |v|) \ \leq \ \Re \tilde{a}(A, B)(u, v).
$$
A second application of Theorem 2.21 proves the assertion.
\end{enumerate}
\end{proof}

\section{Main theorem}\label{section_main_theorems}

Let $d \in \N$ with $d>2$ and let
$\Omega \subset \R^{d}$ be a bounded Lipschitz domain.
Let $A \colon \Omega \to \R^{d\times d}$ be a matrix-valued
function whose coefficients $a_{ij} \colon \Omega \to \R$ are
bounded and measurable. We assume that $A$ is uniformly elliptic,
that is, there exists a constant $\alpha>0$ such that
$$ 
\Re \left(  \left( A(x) \xi \right)^{T} \overline{\xi} \right) \geq \alpha \left| \xi \right|^{2}. 
$$  
for every $\xi \in \C^{d}$ and almost every $x \in \Omega$.\\
\ \\
Let the non-local Robin boundary condition be described by an
operator $B \in \mathcal{L}\left( \mathrm{L}^{2}\left( \partial \Omega \right) \right)$.
Suppose there exists a positive operator $C \in \mathcal{L} \left( \mathrm{L}^{2}( \partial \Omega ) \right)$ 
satisfying
\begin{enumerate}[i)]
\item $\left| B u \right| \ \leq \ C \left| u \right|$ for every $u \in \mathrm{L}^{2}( \partial \Omega )$ and
\item $C{\bf 1} \in \mathrm{L}^{\infty}(\partial \Omega)$ where ${\bf 1}$ denotes the constant function 
with value $1$ on $\partial\Omega$.\\
\end{enumerate}

\begin{thm}\label{Main_result}
Under these assumptions the semigroup 
$(\mathrm{e}^{-tL(A, B)})_{t \geq 0}$ is ultracontractive. More precisely, there exist constants 
$K>0$ and $\lambda > 0$ such that
$$ 
\left\| \mathrm{e}^{-tL(A, B)} u \right\|_{\mathrm{L}^{\infty} \left( \Omega \right) } \ \leq \ 
K t^{-\frac{d}{4}} \ \exp\left(t (\lambda + \frac{\alpha}{2}) \right) \
\left\| u \right\|_{\mathrm{L}^{2} \left( \Omega \right)}
$$
for every time $t > 0$ and every $u \in \mathrm{L}^{2}\left( \Omega \right)$.\\
\end{thm}

\begin{cor}\label{Main_cor}
Let $B$ be defined as
$$
Bu (x) \ = \ \int_{\partial \Omega} k(x,y) u(y) \d \sigma(y)
$$
for $u \in \mathrm{L}^{2}\left( \partial \Omega \right)$ and 
$k \in \mathrm{L}^{2} \left( \partial \Omega \times \partial \Omega \right)$ satisfying
$$
\mathop{\operatorname{ess\,sup}}_{x\in\partial\Omega} \int_{\partial \Omega} \left| k(x,y) \right| \d \sigma(y) < \infty.
$$
Then $(\mathrm{e}^{-tL(A, B)})_{t \geq 0}$ is ultracontractive in $\mathrm{L}^{2}\left( \Omega \right)$ at any time $t > 0$.\\
\end{cor}

\begin{proof}
Define the positive operator $C$ by
$$
(Cu)(x) \ = \ \int_{\partial \Omega} |k(x,y)| u(y) \d \sigma(y)
$$ 
for $u \in \mathrm{L}^{2}\left( \partial \Omega \right)$. Since 
$k \in \mathrm{L}^{2} \left( \partial \Omega \times \partial \Omega \right)$,
$C$ is bounded on $\mathrm{L}^{2}(\partial \Omega)$. Moreover $|Bu|\leq C|u|$ and
$$
(C{\bf 1})(x) \ = \ \int_{\partial\Omega}|k(x,y)| \d\sigma(y).
$$
The assumed bound therefore implies that $C{\bf 1} \in L^{\infty}(\partial\Omega)$. 
The assertion now follows from Theorem \ref{Main_result}.\\
\end{proof}

\section{Comparison with the result of Glück and Mui}\label{comparison_Glueck_Mui}

\label{sec:comparison-glueck-mui}

The ultracontractivity result of Jochen Glück and Jonathan Mui
in \cite[Theorem 3.2]{Glueck_Mui_2026} is proved under the assumption that
the boundary operator
$$
B \in \mathcal{L} \left(\mathrm{L}^2( \partial\Omega )\right)
$$
acts continuously on both $\mathrm{L}^{1}( \partial \Omega )$ and
$\mathrm{L}^{\infty}( \partial \Omega )$. Our assumption is different. 
We require the existence of a positive operator
$C \in \mathcal{L} \left(\mathrm{L}^2( \partial\Omega )\right)$ with
\begin{enumerate}[i)]
\item $|Bu| \ \leq \ C|u|$ for every $u \in \mathrm{L}^2( \partial\Omega )$ and

\item $C{\bf 1} \in \mathrm{L}^{\infty} ( \partial\Omega )$.\\
\end{enumerate}
In this section we compare these two conditions. First note that our assumption 
implies that $B$ acts continuously on $\mathrm{L}^{\infty} ( \partial\Omega )$. 
Indeed, for $u \in \mathrm{L}^{\infty}( \partial\Omega ) \subset \mathrm{L}^2( \partial\Omega )$ 
we have
$$
|Bu| \ \leq \ C|u| \ \leq \ \|u\|_{ \mathrm{L}^{\infty}( \partial\Omega )} C {\bf 1}.
$$
Consequently,
$$
\|Bu\|_{\mathrm{L}^{\infty}( \partial\Omega )} \leq  \| C {\bf 1} \|_{\mathrm{L}^{\infty}( \partial\Omega )}
\|u\|_{ \mathrm{L}^{\infty}( \partial\Omega )}.
$$
Thus the restriction of $B$ to $\mathrm{L}^{\infty}( \partial\Omega )$ is continuous.
In contrast, our assumption does not require $B$ to act continuously on
$\mathrm{L}^{1}( \partial\Omega )$.\\

\begin{prp}
\label{prop:comparison-glueck-mui}
Suppose that $B \in \mathcal{L} \left(\mathrm{L}^2( \partial\Omega )\right)$ acts continuously on
$\mathrm{L}^{1}( \partial\Omega )$ and $\mathrm{L}^{\infty}( \partial\Omega )$. Then there exists a
positive operator $C \in \mathcal{L} \left(\mathrm{L}^2( \partial\Omega )\right)$ such that
$$
|Bu| \ \leq \ C|u|
$$ 
for every $u \in \mathrm{L}^2( \partial\Omega )$ and $C{\bf 1} \in \mathrm{L}^{\infty} ( \partial\Omega )$.\\
\end{prp}

\begin{proof}
This follows from \cite[Lemma~3.1(iii)]{Glueck_Mui_2026}. Indeed, the
positive dominating operator constructed there acts boundedly on
$\mathrm{L}^{\infty}(\partial\Omega)$ and therefore maps ${\bf 1}$ into
$\mathrm{L}^{\infty}(\partial\Omega)$.
\end{proof}

The converse of Proposition \ref{prop:comparison-glueck-mui} is false.
The following simple example shows this.\\

\begin{ex}
\label{ex:strict-comparison-glueck-mui}
Choose a positive function
$b \in \mathrm{L}^{2}(\partial\Omega) \setminus \mathrm{L}^{\infty} (\partial\Omega)$
and define
$$
Bu = Cu =\left( \int_{\partial\Omega} b(y)u(y) \d \sigma(y) \right) {\bf 1}
$$
for $u \in \mathrm{L}^{2}( \partial\Omega )$.
Then $C=B$ is a positive bounded operator on $\mathrm{L}^{2}(\partial\Omega)$. 
Since $\sigma(\partial\Omega) < \infty$, we have $b \in \mathrm{L}^{1}(\partial\Omega)$.
Therefore
$$
C{\bf 1}  = \left( \int_{\partial\Omega} b(y) \d \sigma(y) \right) {\bf 1} \in \mathrm{L}^{\infty} ( \partial\Omega ).
$$
Hence the assumptions of Theorem \ref{Main_result} are satisfied.\\
\ \\
However, $B$ does not act continuously on $\mathrm{L}^{1}(\partial\Omega)$. Otherwise
the functional
$$
 \left\{  u \longmapsto \int_{\partial\Omega}b(y)u(y) \d\sigma(y) \right\}
$$
would be continuous on $\mathrm{L}^{1}( \partial\Omega )$. The identity
$(\mathrm{L}^1(\partial\Omega))^{\ast} = \mathrm{L}^{\infty}( \partial\Omega )$ would then imply
$b \in \mathrm{L}^{\infty}( \partial\Omega )$ which is a contradiction. Thus this
operator is covered by our theorem, but not by Theorem 3.2 in \cite{Glueck_Mui_2026}.
\end{ex}

Proposition \ref{prop:comparison-glueck-mui} and
Example \ref{ex:strict-comparison-glueck-mui} show that, 
in the dimensions considered here and under the common 
assumptions on $\Omega$ and $A$, Theorem \ref{Main_result} applies to a
strictly larger class of boundary operators. The price for the weaker
assumption is that our theorem only gives ultracontractivity of
$e^{-tL(A,B)}$. Glück and Mui also obtain the corresponding estimate
for the adjoint semigroup. Under our approach, the corresponding 
conclusion for the adjoint semigroup follows if one additionally assumes 
the existence of a positive operator
$\tilde{C} \in \mathcal{L} \left( \mathrm{L}^{2} (\partial\Omega) \right)$
such that
\begin{enumerate}[i)]
\item $|B^{\ast} u| \leq \tilde{C}|u|$ for every $u\in L^{2}(\partial\Omega)$,

\item $\tilde{C}{\bf 1} \in \mathrm{L}^{\infty}(\partial\Omega)$.
\end{enumerate}

Indeed,
$$
    L(A,B)^{\ast}=L(A^{T},B^{\ast}),
$$
and Theorem \ref{Main_result} can therefore be applied to $A^{T}$ and $B^{\ast}$.\\
\ \\
For the kernel operators in Corollary \ref{Main_cor} the difference can be seen
directly. The condition
$$
    \mathop{\operatorname{ess\,sup}}_{x\in\partial\Omega}
    \int_{\partial\Omega}|k(x,y)| \d \sigma(y) < \infty
$$
controls the action from $\mathrm{L}^{\infty}( \partial \Omega )$ to itself. A
corresponding standard condition for the action on
$\mathrm{L}^{1}(\partial\Omega)$ would be
$$
    \mathop{\operatorname{ess\,sup}}_{y\in\partial\Omega}
    \int_{\partial\Omega}|k(x,y)| \d \sigma(x) < \infty.
$$
For the kernel $k(x,y)=b(y)$ from Example~\ref{ex:strict-comparison-glueck-mui}, the condition in
Corollary~\ref{Main_cor} is satisfied, whereas $B$ does not act
continuously on $L^1(\partial\Omega)$.\\
\ \\
This difference is also reflected in the proofs. Glück and Mui first
obtain bounds on $\mathrm{L}^{1}(\Omega)$ and $\mathrm{L}^{\infty}(\Omega)$ for the semigroup
and its adjoint. 
In our proof, the bounded eigenfunction and the
domination result first yield an ${L}^{\infty}(\Omega)$-bound for
$\mathrm{e}^{-tL(A,B)}$. 
By duality, this gives an $\mathrm{L}^{1}(\Omega)$-bound for the
adjoint semigroup. Nash's inequality is then applied to the adjoint
semigroup and a second duality argument yields the asserted
$\mathrm{L}^{2}(\Omega)$-to-$\mathrm{L}^{\infty}(\Omega)$ estimate. 
This one-sided argument is the reason why boundedness of $B$ on 
$\mathrm{L}^{1}(\partial\Omega)$ is not required.\\
\ \\
For $d>2$, both results give the same short-time order
$t^{-d/4}$. Neither result yields
$$
\|\mathrm{e}^{-tL(A,B)}\|_{L^{2} \to L^{\infty}} \to 0
$$
as $t \to \infty$. The main
improvement of Theorem \ref{Main_result} is therefore the weaker assumption
on the boundary operator, rather than a stronger conclusion about
the large-time behaviour.

\section{Proof of Theorem \ref{Main_result}}\label{proof_main_theorem}

Throughout this section, let $\Omega$, $A$, $B$ and $C$ satisfy
the assumptions of Theorem \ref{Main_result}.  We use the forms 
$a(A,B)$ and $a(A,-C)$ and their associated operators $L(A,B)$ and 
$L(A,-C)$ as defined in Subsection \ref{form_a}. The key ingredient 
is an eigenfunction $\phi$ of $L(A,-C)$ satisfying
\begin{equation}\label{bounded_eigenfunction}
0<\delta\leq\phi\leq M<\infty
\end{equation}
almost everywhere in $\Omega$. This eigenfunction is constructed
independently in Section \ref{bounded_Eigenfunction} below.

\subsection{An $\mathrm{L}^{1}$-bound for $\mathrm{e}^{-tL(A, B)^{\ast}}$}

We first use (\ref{bounded_eigenfunction}) to derive an $\mathrm{L}^{\infty}$-bound.
Moreover Lemma \ref{inequ_semigroup_B} implies that
\begin{equation}\label{L^infty_continuity}
\left| \mathrm{e}^{-tL(A, B)}u \right| \ \leq \ \mathrm{e}^{-tL(A,-C)} \left| u \right|
\end{equation}
for every $t > 0$ and every $u \in \mathrm{L}^{2}(\Omega)$. We use 
(\ref{bounded_eigenfunction}) to show that the positive comparison semigroup
$\mathrm{e}^{-tL(A,-C)}$ acts boundedly on $\mathrm{L}^{\infty}(\Omega)$.\\

\begin{prp}\label{L_infty_boundedness}
Let $\phi$ be an eigenfunction of $L(A,-C)$ corresponding to an eigenvalue $\mu \in \R$, that is
$$
L(A,-C) \phi \ = \ \mu \phi.
$$
Suppose, in addition, that there exist constants $\delta, M > 0$ such that 
$$
0 \ < \ \delta \ \ \leq \ \phi \ \leq M \ < \ \infty
$$
almost everywhere in $\Omega$. Then the comparison semigroup $(\mathrm{e}^{-tL(A,-C)})_{t \geq 0}$
acts boundedly on $\mathrm{L}^{\infty}(\Omega)$ with
$$
\left\| \mathrm{e}^{-tL(A,-C)} \right\|_{\mathrm{L}^{\infty} \to \mathrm{L}^{\infty}} \ \leq \ \frac{M}{\delta} \mathrm{e}^{-t\mu}.
$$
It follows directly from  (\ref{L^infty_continuity}) that $\mathrm{e}^{-tL(A,B)}$ acts boundedly on
$\mathrm{L}^{\infty}(\Omega)$ for every $t > 0$.\\
\end{prp}

\begin{proof}
For every $u \in \mathrm{L}^{\infty}(\Omega)$, we have
$$
\left| u \right| \ \leq \ \frac{1}{\delta} \left\| u \right\|_{\mathrm{L}^{\infty}(\Omega)} \phi.
$$
Using positivity and the eigenvalue equation, we obtain
$$
\begin{aligned}
\left| \mathrm{e}^{-tL(A,-C)}  u \right|  & \leq \mathrm{e}^{-tL(A,-C)} \left| u \right|  \\
& \leq  \frac{1}{\delta} \left\| u \right\|_{\mathrm{L}^{\infty}(\Omega)} \mathrm{e}^{-tL(A,-C)} \phi \\
& =\frac{1}{\delta} \mathrm{e}^{-t\mu} \left\| u \right\|_{\mathrm{L}^{\infty}(\Omega)} \phi \\
& \leq  \frac{M}{\delta} \mathrm{e}^{-t\mu} \left\| u \right\|_{\mathrm{L}^{\infty}(\Omega)}.
\end{aligned}
$$
\end{proof}

In the next lemma, we use the Hahn-Banach extension theorem to show that $\mathrm{e}^{-tL(A, B)^{\ast}}$
extends to a bounded operator on  $\mathrm{L}^{1}\left( \Omega \right)$ at any time $t > 0$.\\

\begin{lem}\label{L1_Linfty_semigroup}
For every $t>0$ the operator $\mathrm{e}^{-tL(A, B)^{\ast}}$ extends uniquely
to a bounded operator on $\mathrm{L}^{1}(\Omega)$ and
$$
\left\| \mathrm{e}^{-tL(A, B)^{\ast}} \right\|_{\mathrm{L}^{1} \to \mathrm{L}^{1}} \ \leq \ 
\left\| \mathrm{e}^{-tL(A, B)} \right\|_{ \mathrm{L}^{\infty} \to \mathrm{L}^{\infty} } \ \leq\ \frac{M}{\delta} \mathrm{e}^{-t\mu}.
$$
\end{lem}

\begin{proof}
Let $u \in \mathrm{L}^{2}\left( \Omega \right)$. Since $\Omega$ has finite measure, 
$\mathrm{L}^{2}\left( \Omega \right)$ embeds continuously into $\mathrm{L}^{1}\left( \Omega \right)$. 
Hence $\mathrm{e}^{-tL(A, B)^{\ast}}u \in \mathrm{L}^{2}\left( \Omega \right)$ also 
belongs to $\mathrm{L}^{1}\left( \Omega \right)$. By the Hahn-Banach theorem, 
there exists 
$v \in \mathrm{L}^{\infty}( \Omega ) = \left( \mathrm{L}^{1}( \Omega ) \right)^{\ast}$ with 
$\| v \|_{\mathrm{L}^{\infty}(\Omega)} = 1$ such that
$$
\langle \mathrm{e}^{-tL(A, B)^{\ast}}u, v \rangle \ = \ \left\| \mathrm{e}^{-tL(A, B)^{\ast}}u \right\|_{\mathrm{L}^{1}(\Omega)}.
$$
Since $\Omega$ has finite measure, 
$v \in \mathrm{L}^{\infty} \left( \Omega \right) \subset \mathrm{L}^{2} \left( \Omega \right)$.
Hence we obtain
\begin{align*}
\left\| \mathrm{e}^{-tL(A, B)^{\ast}}u \right\|_{\mathrm{L}^{1}(\Omega)} & = 
\left|  \langle \mathrm{e}^{-tL(A, B)^{\ast}}u, v \rangle \right| = \left|  \langle u, \mathrm{e}^{-tL(A, B)}v \rangle \right| \\
\ \\
& \leq \left\| \mathrm{e}^{-tL(A, B)}v \right\|_{\mathrm{L}^{\infty}(\Omega)} \ \| u \|_{\mathrm{L}^{1}(\Omega)}
\leq \left\| \mathrm{e}^{-tL(A, B)} \right\|_{ \mathrm{L}^{\infty} \to \mathrm{L}^{\infty} } \
\left\| v \right\|_{\mathrm{L}^{\infty}(\Omega)} \ \| u \|_{\mathrm{L}^{1}(\Omega)} \\
\ \\
& = \left\| \mathrm{e}^{-tL(A, B)} \right\|_{ \mathrm{L}^{\infty} \to \mathrm{L}^{\infty} } \ \| u \|_{\mathrm{L}^{1}(\Omega)} 
\end{align*}
where we used that $\mathrm{e}^{-tL(A, B)}$ is bounded on $\mathrm{L}^{\infty}\left( \Omega \right)$
and that $ \mathrm{e}^{-tL(A, B)^{\ast}}$ is the adjoint of $\mathrm{e}^{-tL(A, B)}$ on $\mathrm{L}^{2}\left( \Omega \right)$.
Therefore, $\mathrm{e}^{-tL(A, B)^{\ast}}$ extends uniquely to a bounded operator on 
$\mathrm{L}^{1}(\Omega)$, with
$$
\left\| \mathrm{e}^{-tL(A, B)^{\ast}} \right\|_{\mathrm{L}^{1} \to \mathrm{L}^{1}} \ \leq \ 
\left\| \mathrm{e}^{-tL(A, B)} \right\|_{ \mathrm{L}^{\infty} \to \mathrm{L}^{\infty} }
$$
since $\mathrm{L}^{2}(\Omega)$ is dense in $\mathrm{L}^{1}(\Omega)$.\\
\end{proof}

\subsection{The $\mathrm{L}^{1}$-to-$\mathrm{L}^{2}$ estimate for $\mathrm{e}^{-tL(A, B)^{\ast}}$}\label{Second_part}

We show that, for every $t>0$, the operator $\mathrm{e}^{-tL(A, B)^{\ast}}$ extends to a bounded operator from 
$\mathrm{L}^{1}\left( \Omega \right)$ to $\mathrm{L}^{2}\left( \Omega \right)$. 
The main ingredients are Nash's inequality (\ref{Nash_inequality}) and the estimate 
$$
\Re \tilde{a}(A, B)^{\ast}(v,v) = \Re \tilde{a}(A, B)(v,v) \geq  \frac{\alpha}{2} \| \nabla v \|_{\mathrm{L}^{2}(\Omega)^{d}}^{2}
$$
for every $v \in D \left( \tilde{a}(A, B)^{\ast} \right) = D \left( \tilde{a}(A, B) \right)$, as follows from Remark \ref{a(B)_Nash}. 
Furthermore, by Lemma \ref{L1_Linfty_semigroup}, $\mathrm{e}^{-tL(A, B)^{\ast}}$ extends to a bounded operator on
$\mathrm{L}^{1}\left( \Omega \right)$ for every $t>0$ with
$$
\left\| \mathrm{e}^{-tL(A, B)^{\ast}} \right\|_{\mathrm{L}^{1} \to \mathrm{L}^{1}} \ \leq \ \frac{M}{\delta} \mathrm{e}^{-t\mu}.
$$
We use the following version of Nash's inequality:
\begin{equation} \label{Nash_inequality}
\exists C_{d} > 0 \ \forall v \in H^{1}( \Omega ) \ : \
\| v \|_{\mathrm{L}^{2}(\Omega)}^{2+\frac{4}{d}} \leq C_{d} \| v \|_{\mathrm{L}^{1}(\Omega)}^{\frac{4}{d}}
\| v \|_{H^{1}(\Omega)}^{2};
\end{equation}
see \cite[Theorem 12.3.1, p. 162]{Arendt06}. If necessary, increase the common shift $\lambda$ from
Remark \ref{choice_lambda} so that
$$
    \mu + \lambda + \frac{\alpha}{2} \geq 0.
$$

\begin{lem} \label{adjoint_semigroup_L1_L2}
For any $t > 0$ the operator $\mathrm{e}^{-tL(A, B)^{\ast}}$ 
extends uniquely to a bounded operator from 
$\mathrm{L}^{1}\left( \Omega \right)$ to $\mathrm{L}^{2}\left( \Omega \right)$. In particular, there exists
a constant $K > 0$ such that
$$
\| \mathrm{e}^{-tL(A, B)^{\ast}} u \|_{\mathrm{L}^{2}\left( \Omega \right) } \ \leq \ 
K t^{-d/4} \ \exp\left(t (\lambda + \frac{\alpha}{2}) \right) \ \left\| u \right\|_{\mathrm{L}^{1} \left( \Omega \right)}
$$
for every $u \in \mathrm{L}^{1} \left( \Omega \right)$ and $t>0$.\\
\end{lem}

\begin{proof}
\begin{enumerate}[i)]
\item By the definition of the adjoint form
$$
\tilde{a}(A,B)^{\ast}(u,v) = \overline{\tilde{a}(A,B)(v,u)}.
$$
Consequently,
$$
\Re \tilde{a}(A,B)^{*}(u,u) = \Re \tilde{a}(A,B)(u,u).
$$
Thus, by Remark \ref{a(B)_Nash},
\begin{equation} \label{Ellipticity_A}
\Re \tilde{a}(A, B)^{\ast}(u,u) \ = \ \Re \tilde{a}(A, B)(u,u) \ \geq \ \frac{\alpha}{2} \| \nabla u \|_{\mathrm{L}^{2}(\Omega)^{d}}^{2}
\end{equation}
for every $u \in D \left( \tilde{a}(A, B)^{\ast} \right)$.

\item We now apply Nash's inequality. The semigroup associated with the adjoint form $\tilde{a}(A, B)^{\ast}$
is
$$
(\mathrm{e}^{-t\tilde{L}(A, B)^{\ast}})_{t \geq 0}.
$$
Let $u \in D ( \tilde{L}(A, B)^{\ast} )$ and fix $t > 0$. We define
$$
v(t) = \mathrm{e}^{-\frac{\alpha t}{2}} \mathrm{e}^{-t\tilde{L}(A, B)^{\ast}} u \in 
D ( \tilde{L}(A, B)^{\ast} ) \subset \mathrm{L}^{2} \left( \Omega \right).
$$
Since $\Omega$ has finite measure, both $v(t)$ and $u$ belong to $\mathrm{L}^{1} \left( \Omega \right)$.
By Proposition \ref{Semigroup_adjustment}, we have
$$
\mathrm{e}^{-t\tilde{L}(A, B)^{\ast}} = \mathrm{e}^{-t\lambda}\mathrm{e}^{-tL(A, B)^{\ast}}.
$$
By Lemma \ref{L1_Linfty_semigroup},
$$
\left\| v(t) \right\|_{\mathrm{L}^{1} \left( \Omega \right)} \ \leq \ \mathrm{e}^{- t (\lambda + \frac{\alpha}{2})} \ 
\left\| \mathrm{e}^{-tL(A, B)^{\ast}} \right\|_{\mathrm{L}^{1} \to \mathrm{L}^{1}} \ 
\left\| u \right\|_{\mathrm{L}^{1} \left( \Omega \right)}
\leq \ \frac{M}{\delta} \exp \left(-t(\mu + \lambda + \frac{\alpha}{2}) \right) 
\left\| u \right\|_{\mathrm{L}^{1} \left( \Omega \right)}.
$$
Since $u \in D( \tilde{L}(A, B)^{\ast} )$, standard semigroup theory implies that the map $s \mapsto v(s)$ is 
continuously differentiable in $\mathrm{L}^{2} \left( \Omega \right)$. Therefore,
\begin{align*}
\frac{\partial}{\partial t} \| v(t) \|^{2}_{\mathrm{L}^{2}\left( \Omega \right)} & = 
2 \Re \langle v^{\prime}(t), v(t) \rangle_{\mathrm{L}^{2}\left( \Omega \right)}
 = - 2 \Re \langle \tilde{L}(A, B)^{\ast} v(t), v(t)  \rangle_{\mathrm{L}^{2}\left( \Omega \right)} 
- \alpha \left\| v(t) \right\|_{\mathrm{L}^{2}\left( \Omega \right)}^{2} & \\
& = - \left( 2 \Re  \tilde{a}(A, B)^{\ast}(v(t), v(t)) + \alpha \left\| v(t) \right\|_{\mathrm{L}^{2}\left( \Omega \right)}^{2} \right) & \\
& \leq - \alpha  \| v(t) \|_{H^{1}(\Omega)}^{2} &
\end{align*}
by (\ref{Ellipticity_A}). If $v(t)=0$, the desired estimate is trivial. Hence we may assume
that $v(t) \neq 0$. By the semigroup property, this also implies $v(s) \neq 0$ for every $s\in[0,t]$.
Now we use Nash's inequality (\ref{Nash_inequality}) to obtain
$$
\frac{\partial}{\partial t} \| v(t) \|^{2}_{\mathrm{L}^{2}\left( \Omega \right)} \ \leq \
 - \alpha  \| v(t) \|_{H^{1}(\Omega)}^{2} \ \leq \ - \frac{\alpha}{C_{d}}  
\frac{ \| v(t) \|_{\mathrm{L}^{2}(\Omega)}^{2+\frac{4}{d}} }{ \| v(t) \|_{\mathrm{L}^{1}(\Omega)}^{\frac{4}{d}}  }.
$$
The chain rule therefore gives 
\begin{align*}
\frac{\partial}{\partial t} \left( \| v(t) \|^{2}_{\mathrm{L}^{2}\left( \Omega \right)} \right)^{-\frac{2}{d}} 
\ & \geq \ \left( -\frac{2}{d}  \right) \left( \| v(t) \|^{2}_{\mathrm{L}^{2}\left( \Omega \right)} \right)^{-\frac{2}{d}-1}
 \left( - \frac{\alpha}{C_{d}} \right)  
\frac{ \| v(t) \|_{\mathrm{L}^{2}(\Omega)}^{2+\frac{4}{d}} }{ \| v(t) \|_{\mathrm{L}^{1}(\Omega)}^{\frac{4}{d}}  } \\
\ \\
& = \ \frac{2 \alpha}{d C_{d}} \ \| v(t) \|_{\mathrm{L}^{1}(\Omega)}^{-\frac{4}{d}} & \\
\ \\
& \geq \ \frac{2 \alpha }{d C_{d}} \ \left( \frac{\delta}{M} \right)^{4/d} \exp \left( \frac{4t}{d}(\mu + \lambda  + \frac{\alpha}{2} ) \right) 
\left\| u \right\|_{\mathrm{L}^{1} \left( \Omega \right)}^{-\frac{4}{d}}. &
\end{align*}
Integrating over $[0,t]$, we obtain
\begin{align*}
\| v(t) \|_{\mathrm{L}^{2}\left( \Omega \right)}^{-\frac{4}{d}} 
& = \ \int_{0}^{t}  \frac{\partial}{\partial s} \ \| v(s) \|_{\mathrm{L}^{2}\left( \Omega \right)}^{-\frac{4}{d}} \d s 
+ \| u \|_{\mathrm{L}^{2}\left( \Omega \right)}^{-\frac{4}{d}} & \\
\ \\
& \geq \  \frac{2 \alpha }{d C_{d}} \ \left( \frac{\delta}{M} \right)^{4/d}  
\int_{0}^{t} \exp \left( \frac{4s}{d}(\mu + \lambda  + \frac{\alpha}{2} ) \right) \d s 
\left\| u \right\|_{\mathrm{L}^{1} \left( \Omega \right)}^{-\frac{4}{d}} &  \\
\ \\
& \geq \ \frac{2 \alpha }{d C_{d}} \ \left( \frac{\delta}{M} \right)^{4/d} \int_{0}^{t} 1 \d s 
\left\| u \right\|_{\mathrm{L}^{1} \left( \Omega \right)}^{-\frac{4}{d}} &  \\
\ \\
& =  \ \frac{2 \alpha }{d C_{d}} \left( \frac{\delta}{M} \right)^{4/d} 
\left\| u \right\|_{\mathrm{L}^{1} \left( \Omega \right)}^{-\frac{4}{d}} t & 
\end{align*} 
since $\mu + \lambda  + \frac{\alpha}{2} \geq 0$ by the choice of $\lambda$. Consequently,
$$
\| \mathrm{e}^{-\frac{\alpha t}{2}} \mathrm{e}^{-t\tilde{L}(A, B)^{\ast}} u \|_{\mathrm{L}^{2}\left( \Omega \right) } \ = \ 
\| v(t) \|_{\mathrm{L}^{2}\left( \Omega \right) } \ \leq \
\frac{M}{\delta} \left( \frac{d C_{d}}{2 \alpha} \right)^{d/4} \ t^{-d/4} \ 
\left\| u \right\|_{\mathrm{L}^{1} \left( \Omega \right)}
$$
Consequently, for every $u \in D ( \tilde{L}(A, B)^{\ast} )$,
$$
\| \mathrm{e}^{-t\tilde{L}(A, B)^{\ast}} u \|_{\mathrm{L}^{2}\left( \Omega \right) } \ \leq \
\frac{M}{\delta} \left( \frac{d C_{d}}{2 \alpha} \right)^{d/4} \ t^{-d/4} \ \mathrm{e}^{\frac{\alpha t}{2}} \
\left\| u \right\|_{\mathrm{L}^{1} \left( \Omega \right)}.
$$

\item Since $-\tilde{L}(A,B)^{\ast}$ generates a $C_{0}$-semigroup,
$D(\tilde{L}(A,B)^{\ast})$ is dense in $\mathrm{L}^{2}(\Omega)$.
Moreover, it is dense in $\mathrm{L}^{1}(\Omega)$.
Indeed, every element of $\mathrm{L}^{1}(\Omega)$
can first be approximated in $\mathrm{L}^{1}(\Omega)$ 
by elements of $\mathrm{L}^{2}(\Omega)$. Each of these elements 
can then be approximated in $\mathrm{L}^{2}(\Omega)$
and hence also in $\mathrm{L}^{1}(\Omega)$ by elements of $D(\tilde{L}(A,B)^{\ast})$.\\
Therefore, for every $t>0$, $\mathrm{e}^{-tL(A,B)^{\ast}}$ has a unique bounded extension from 
$\mathrm{L}^{1}(\Omega)$ to $\mathrm{L}^{2}(\Omega)$ satisfying
$$
\exists K > 0 \ \forall t > 0 \ \forall u \in \mathrm{L}^{1} \left( \Omega \right) \ : \
\| \mathrm{e}^{-tL(A, B)^{\ast}} u \|_{\mathrm{L}^{2}\left( \Omega \right) } \ \leq \ 
K t^{-d/4} \ \exp\left(t (\lambda + \frac{\alpha}{2}) \right) \ \left\| u \right\|_{\mathrm{L}^{1} \left( \Omega \right)},
$$
where we used $\mathrm{e}^{-t\tilde{L}(A, B)^{\ast}} = \mathrm{e}^{-t\lambda} \mathrm{e}^{-tL(A, B)^{\ast}}$.
\end{enumerate}
\end{proof}

\subsection{Duality}

For every $t > 0$ let 
$$
T(t) \ \colon \ \mathrm{L}^{2}\left( \Omega \right) \to \mathrm{L}^{\infty}\left( \Omega \right)
$$ 
be the Banach space adjoint of the bounded extension 
$$
\mathrm{e}^{-tL(A, B)^{\ast}} \ \colon \ \mathrm{L}^{1}\left( \Omega \right) \to \mathrm{L}^{2}\left( \Omega \right)
$$
obtained in Lemma \ref{adjoint_semigroup_L1_L2}. Then
$$
\langle \mathrm{e}^{-tL(A, B)^{\ast}}u, v  \rangle = \langle u, T(t)v \rangle
$$
for every $u \in \mathrm{L}^{1}\left( \Omega \right)$ and $v \in \mathrm{L}^{2}\left( \Omega \right)$, and
$$
\|  T(t) \|_{\mathrm{L}^{2} \to \mathrm{L}^{\infty}} = \|  \mathrm{e}^{-tL(A, B)^{\ast}} \|_{\mathrm{L}^{1} \to \mathrm{L}^{2}} 
\leq K t^{-d/4} \ \exp\left(t (\lambda + \frac{\alpha}{2}) \right).
$$
Since $L^{2}(\Omega)\subseteq L^{1}(\Omega)$ and
$\left( \mathrm{e}^{-tL(A,B)}\right)^{\ast} = \mathrm{e}^{-tL(A,B)^{\ast}}$
we obtain
$$
\langle \mathrm{e}^{-tL(A, B)} u, v \rangle_{\mathrm{L}^{2}(\Omega)} 
= \langle u, \mathrm{e}^{-tL(A, B)^{\ast}}v \rangle_{\mathrm{L}^{2}(\Omega)} = \langle T(t)u, v \rangle_{\mathrm{L}^{2}(\Omega)}  
$$
for all $u,v\in \mathrm{L}^{2}(\Omega)$. Hence
$$
\mathrm{e}^{-tL(A, B)} u = T(t)u
$$
for every $u \in \mathrm{L}^{2} \left( \Omega \right)$. The estimate for $T(t)$ therefore
proves Theorem \ref{Main_result}.

\section{Two-sided bounds for a positive eigenfunction of $L(A,-C)$}\label{bounded_Eigenfunction}

In this section we focus on heat equations
$$
\frac{\partial u}{\partial t} - \operatorname{div} \left( A \nabla u \right) = 0
$$ 
on bounded Lipschitz domains $\Omega$ in $\R^{d}$ for $d>2$, where $-\operatorname{div} \left( A \nabla \cdot \right)$ 
is a second-order uniformly elliptic operator with non-local Robin boundary conditions formally given by 
$$
\nu \cdot A \nabla u - Cu = 0
$$ 
with an outer unit normal $\nu$ on $\partial\Omega$, a positive and continuous boundary operator 
$C \in \mathcal{L} \left( \mathrm{L}^{2}\left( \partial \Omega  \right) \right)$ satisfying 
$C{\bf 1} \in \mathrm{L}^{\infty}(\partial \Omega)$
and a matrix-valued function $A \colon \Omega \to \R^{d \times d}$ with bounded and measurable
coefficients $a_{ij} \colon \Omega \to \R$. Furthermore, $A$ is uniformly elliptic on 
$\Omega$, i.e.
$$
\exists \alpha > 0 \ \forall \xi \in \C^{d}, x \in \Omega \ : \ 
\Re \left(  \left( A(x) \xi \right)^{T} \overline{\xi} \right) \geq \alpha \left| \xi \right|^{2}. 
$$
\ \\
According to Subsection \ref{form_a} an associated sesquilinear 
form $a(A,-C)$ to the mentioned heat equation with non-local Robin boundary conditions is defined by
\begin{equation}
a(A, -C)(u,v) = \int_{\Omega} A \nabla u \cdot \overline{\nabla v} \d x - 
\int_{\partial \Omega} C \gamma(u) \ \overline{\gamma(v)} \d \sigma(x)
\end{equation}
for $u,v \in D \left( a(A, -C) \right) = H^{1}(\Omega)$. The associated operator to $a(A,-C)$ is noted by
$L(A,-C)$ and $( \mathrm{e}^{-tL(A,-C)} )_{t \geq 0}$ is the associated operator semigroup in $\mathrm{L}^{2}(\Omega)$.\\
\ \\
In this section we prove the existence of an eigenfunction $\phi$ of $L(A,-C)$ such that 
$$
\exists \delta, M > 0 \ : \ \delta \ \leq \ \phi \ \leq \ M 
$$
is true almost everywhere on $\Omega$. Our argumentation is divided into these steps: 
\begin{enumerate}
\item Positivity and compactness of the resolvents of $L(A,-C)$, 
\item Existence of a positive eigenfunction $\phi$ of $L(A,-C)$ by Krein-Rutman,
\item Upper bound for $\phi$ by iteration of exponents, 
\item Lower bound for $\phi$ by comparison with the Neumann semigroup.\\ 
\end{enumerate}

\begin{rem}
Notice that the property $C{\bf 1} \in \mathrm{L}^{\infty}( \partial \Omega)$ is not needed to prove the existence of a strictly
positive eigenfunction $\phi$ nor is it needed to prove a strictly positive lower bound $\delta$ of $\phi$. It is exclusively needed
to prove the existence of the upper bound $M$.
\end{rem}

\subsection{Preliminaries}\label{Neumann_semigroups_introduction}

We focus our argumentation on Neumann forms $a(A, 0)$ and Neumann semigroups  $( \mathrm{e}^{-tL(A,0)} )_{t \geq 0}$ in 
$\mathrm{L}^{2}(\Omega)$. According to Subsection \ref{form_a} these forms are defined as
\begin{equation}
a(A, 0)(u,v) = \int_{\Omega} A \nabla u \cdot \overline{\nabla v} \d x
\end{equation}
for $u,v \in D \left( a(A, 0) \right) = H^{1}\left( \Omega \right)$. The negative associated operator $-L(A,0)$ 
of $a(A, 0)$ generates $( \mathrm{e}^{-tL(A,0)} )_{t \geq 0}$.
The operators $-L(A,-C)$ and $-L(A,0)$ are generators of $C_{0}$-semigroups in $\mathrm{L}^{2}(\Omega)$. Hence for a sufficiently
large $\lambda > 0$ we conclude that $\lambda$ is contained in both resolvent sets $\rho\left( -L(A,0) \right)$
and $\rho\left( -L(A,-C) \right)$.\\

\begin{lem}\label{resolvent_comparison}
Let $\lambda > 0$ be sufficiently large such that $\lambda \in \rho\left( -L(A,0) \right) \cap \rho\left( -L(A,-C) \right)$ holds. Then
$$
0 \ \leq \ \left( \lambda + L(A,0) \right)^{-1} u \ \leq \ \left( \lambda + L(A,-C) \right)^{-1} u
$$
is true almost everywhere in $\Omega$ for any $0 \leq u \in \mathrm{L}^{2}(\Omega)$.
\end{lem}

\begin{proof}
Let $0 \leq u \in \mathrm{L}^{2}(\Omega)$ be arbitrary but fixed. Then set 
$$
v_{0} = \left( \lambda + L(A,0) \right)^{-1} u \ \text{and} \ v_{C} = \left( \lambda + L(A,-C) \right)^{-1} u.
$$
Both of these functions $v_{0}$ and
$v_{C}$ are positive since their respective operator semigroups consist of positive operators as shown in the proof of 
Lemma \ref{inequ_semigroup_B}. Hence
$$
v_{C} = \left( \lambda + L(A,-C) \right)^{-1} u = \int_{0}^{\infty} \mathrm{e}^{-\lambda t}  \mathrm{e}^{-tL(A,-C)}u \d t \geq 0
$$
and similarly $v_{0} \geq 0$ are true pointwise almost everywhere in $\Omega$. Furthermore
\begin{equation}\label{Neumann_form_equation}
\int_{\Omega} A \nabla v_{0} \cdot \overline{\nabla w} \d x + \lambda \int_{\Omega} v_{0} \overline{w} \d x
 \ = \ \int_{\Omega} u\overline{w} \d x
\end{equation}
and 
\begin{equation}\label{C_form_equation}
\int_{\Omega} A \nabla v_{C} \cdot \overline{\nabla w} \d x + \lambda \int_{\Omega} v_{C} \overline{w} \d x
- \langle C\gamma(v_{C}), \gamma(w) \rangle_{\mathrm{L}^{2}(\partial \Omega)}
 \ = \ \int_{\Omega} u\overline{w} \d x
\end{equation}
are both implied for every $w \in H^{1}(\Omega)$. Notice that $v_{0}$ and $v_{C}$ are both contained in the domain of their
respective forms. Therefore we set
$w = \left( v_{0} - v_{C} \right)^{+} \in H^{1}(\Omega)$. Subtracting (\ref{C_form_equation}) from (\ref{Neumann_form_equation})
implies
$$
0 = \int_{\Omega} A \nabla (v_{0} - v_{C}) \cdot \nabla w \d x + \lambda \int_{\Omega} (v_{0} - v_{C}) w \d x
+ \langle C\gamma(v_{C}), \gamma(w) \rangle_{\mathrm{L}^{2}(\partial \Omega)}
$$
where $C\gamma(v_{C})$ and $\gamma(w)$ are both positive functions on $\partial \Omega$. Furthermore
\begin{align*}
\int_{\Omega} A \nabla (v_{0} - v_{C}) \cdot \nabla w \d x & = \Re \int_{\Omega} A \nabla (v_{0} - v_{C}) \cdot \nabla w \d x \\
\ \\
& = \int_{\Omega} 1_{\{ v_{0} > v_{C} \}} \Re A \nabla (v_{0} - v_{C}) \cdot \nabla (v_{0} - v_{C})  \d x \\ 
\ \\
& \geq \alpha  \int_{\Omega} 1_{\{ v_{0} > v_{C} \}} \left| \nabla (v_{0} - v_{C}) \right|^{2} \d x \\
\ \\
& = \alpha \| \nabla w \|_{\mathrm{L}^{2}(\Omega)^{d}}^{2} \geq 0
\end{align*}
is true where we used the fact that $\nabla (v_{0} - v_{C})$ and $\nabla w$ are both real-valued functions for the first equation. So
\begin{align*}
0 & = \int_{\Omega} A \nabla (v_{0} - v_{C}) \cdot \nabla w \d x + \lambda \int_{\Omega} (v_{0} - v_{C}) w \d x
+ \langle C\gamma(v_{C}), \gamma(w) \rangle_{\mathrm{L}^{2}(\partial \Omega)} \\
& \geq \lambda \int_{\Omega} (v_{0} - v_{C}) w \d x = \lambda \| w \|_{\mathrm{L}^{2}(\Omega)}^{2} \geq 0
\end{align*}
implies $w=0$ almost everywhere in $\Omega$ which proves the claim.\\
\end{proof}

Next let us show that the property $C{\bf 1} \in \mathrm{L}^{\infty}(\partial \Omega)$ of the boundary operator $C$
implies a $\mathrm{L}^{\infty}$-continuity first and then implies a $\mathrm{L}^{p}$-boundedness for every $2 \leq p \leq \infty$
by the Riesz-Thorin interpolation theorem.\\

\begin{lem}\label{C_interpolation}
The boundary operator $C$ is  $\mathrm{L}^{\infty}$-continuous with 
$$
\left\| C \right\|_{\mathrm{L}^{\infty} \to \mathrm{L}^{\infty}} \ = \ \left\| C {\bf 1}\right\|_{\mathrm{L}^{\infty}(\partial \Omega)}.
$$
In particular, 
$$
\left| Cu \right| \ \leq \ \left\| u \right\|_{\mathrm{L}^{\infty}(\partial \Omega)} C {\bf 1}
$$ 
holds pointwise almost everywhere in $\partial \Omega$ for every $u \in \mathrm{L}^{\infty}(\partial \Omega)$. Furthermore
$C$ acts boundedly on $\mathrm{L}^{p}(\partial \Omega)$ for every $p \in [2,\infty)$ with
$$
\left\| C \right\|_{\mathrm{L}^{p} \to \mathrm{L}^{p}} \ \leq \ 
\max\left\{ \left\| C \right\|_{\mathrm{L}^{2} \to \mathrm{L}^{2}}, \left\| C {\bf 1}\right\|_{\mathrm{L}^{\infty}(\partial \Omega)} 
\right\} 
$$
\end{lem}

\begin{proof}
Let $u \in \mathrm{L}^{\infty}(\partial \Omega)$ be arbitrary but fixed. Then $\left| Cu \right| \leq C \left| u \right|$ is implied pointwise
almost everywhere on $\partial \Omega$ by the positivity of the boundary operator $C$. Remember that $Cu$ is well defined since
$\mathrm{L}^{\infty}(\partial \Omega)$ is contained in $\mathrm{L}^{2}(\partial \Omega)$
as $\Omega$ is a bounded Lipschitz domain in $\R^{d}$. Using 
$\left| u\right| \leq \left\| u \right\|_{\mathrm{L}^{\infty}(\partial \Omega)}$ and the positivity of $C$ once again we conclude to
$$
\left| Cu \right| \ \leq \ C \left| u \right| \ \leq 
\left\| u \right\|_{\mathrm{L}^{\infty}(\partial \Omega)} C {\bf 1}
$$ 
which implies a $\mathrm{L}^{\infty}$-continuity of $C$ with 
$\left\| C \right\|_{\mathrm{L}^{\infty} \to \mathrm{L}^{\infty}} \ \leq \ \left\| C {\bf 1}\right\|_{\mathrm{L}^{\infty}(\partial \Omega)}$.
Furthermore 
$$
\left\| C {\bf 1}\right\|_{\mathrm{L}^{\infty}(\partial \Omega)} \leq \left\| C \right\|_{\mathrm{L}^{\infty} \to \mathrm{L}^{\infty}} 
\left\| {\bf 1} \right\|_{\mathrm{L}^{\infty}(\partial \Omega)} = \left\| C \right\|_{\mathrm{L}^{\infty} \to \mathrm{L}^{\infty}} 
$$
proves the claim.
\end{proof}

\subsection{Existence of a positive eigenfunction $\phi$ of $L(A,-C)$}\label{existence_phi}

Our main argument in this subsection is the Krein-Rutman theorem to prove the existence of a 
positive eigenfunction $\phi$ of $L(A,-C)$. For a sufficiently large $0 < \lambda$ we show
that the spectral radius of the resolvent 
$$
R_{\lambda} = \left( \lambda + L(A,-C) \right)^{-1} \in \mathcal{L} \left( \mathrm{L}^{2}(\Omega) \right)
$$
is an eigenvalue of $R_{\lambda}$ to a positive eigenfunction $\phi$. To satisfy the conditions of the 
Krein-Rutman theorem in $\mathrm{L}^{2}(\Omega, \R)$, we must show that $R_{\lambda}$ is a compact 
and positive operator in  $\mathrm{L}^{2}(\Omega, \R)$ and that the spectral radius $\rho(R_{\lambda})$ 
of $R_{\lambda}$ is strictly positive. Please notice that 
$$
\mathcal{K} = \left\{ u \in \mathrm{L}^{2}(\Omega, \R) \ : \ u \geq 0 \ \text{a. e. in} \ \Omega \right\}
$$
is a total cone in $\mathrm{L}^{2}(\Omega, \R)$ and invariant under $R_{\lambda}$.\\

In the following argumentation let $\lambda > 0$ be sufficiently large that there exists a constant 
$c_{\lambda} > 0$ such that
$$
\Re a(A,-C) (u,u) + \lambda \| u \|_{\mathrm{L}^{2}(\Omega, \R)}^{2} \ \geq \ c_{\lambda} \| u \|_{H^{1}(\Omega, \R)}^{2}
$$
is true for every $u \in H^{1}(\Omega, \R)$. Please compare with (\ref{H1_inequ}) in Subsection \ref{form_a}.

\begin{lem}
The resolvent $R_{\lambda}$ is a positive and compact operator in $\mathrm{L}^{2}(\Omega, \R)$.
\end{lem}

\begin{proof}
Due to 
$$
R_{\lambda}v = \int_{0}^{\infty} \mathrm{e}^{-\lambda t}  \mathrm{e}^{-tL(A,-C)}v \d t \geq 0 
$$
for $0 \leq v \in \mathrm{L}^{2}(\Omega, \R)$, we only have to prove the compactness of $R_{\lambda}$. \\
Let $f \in \mathrm{L}^{2}(\Omega, \R)$ be arbitrary but fixed and set 
$$
u = R_{\lambda} f \in D\left( L(A, -C)\right) \subset H^{1}(\Omega, \R).
$$
Then
\begin{align*}
c_{\lambda} \| u \|_{H^{1}(\Omega, \R)}^{2} & \leq \Re a(A,-C) (u,u) + \lambda \| u \|_{\mathrm{L}^{2}(\Omega, \R)}^{2} = 
\Re \langle f, u \rangle_{\mathrm{L}^{2}(\Omega, \R)} \\
& \leq \| f \|_{\mathrm{L}^{2}(\Omega, \R)} \ \| u \|_{H^{1}(\Omega, \R)}
\end{align*}
is true which implies $\| R_{\lambda} f \|_{H^{1}(\Omega, \R)} \leq c_{\lambda}^{-1}  \| f \|_{\mathrm{L}^{2}(\Omega, \R)}$. We define
a bounded operator
$$
\begin{array}{rl} 
S_{\lambda} \colon \mathrm{L}^{2}(\Omega, \R) & \to  H^{1}(\Omega, \R) \\
         f & \mapsto R_{\lambda} f.
\end{array}
$$ 
The Rellich-Kondrachov embedding theorems state that the identity operator $\operatorname{Id}$ is a compact embedding 
from $ H^{1}(\Omega, \R)$ into $\mathrm{L}^{2}(\Omega, \R)$. So $R_{\lambda} = \operatorname{Id} \circ S_{\lambda}$
is a compact operator in $\mathrm{L}^{2}(\Omega, \R)$.\\
\end{proof}

Next we show that the spectral radius $\rho(R_{\lambda})$ of $R_{\lambda}$ is strictly positive. Please note that 
$R_{\lambda}$ being a compact operator and $\mathrm{L}^{2}(\Omega, \R)$ being infinite-dimensional implies 
$0 \in \sigma(R_{\lambda})$. So the spectrum of $R_{\lambda}$ is not empty and hence it is reasonable to study the
spectral radius $\rho(R_{\lambda})$ of $R_{\lambda}$ which is defined as
$$
\rho(R_{\lambda}) = \sup \left\{ |\mu| \ : \ \mu \in \sigma(R_{\lambda}) \right\}.
$$
Furthermore $\rho(R_{\lambda})$ is further characterized by 
$\rho(R_{\lambda}) = \lim_{n \to \infty} \| R_{\lambda}^{n} \|_{\mathrm{L}^{2} \to \mathrm{L}^{2}}^{1/n}$.\\

\begin{lem}\label{1_Neumann_eigenvalue} 
\begin{enumerate}[i)] Let $\lambda > 0$ be sufficiently large with
$\lambda \in \rho\left( -L(A,0) \right) \cap \rho\left( -L(A,-C) \right)$.
\item Then ${\bf 1}$ is an eigenfunction to the eigenvalue $\lambda^{-1} > 0$ of the Neumann resolvent 
$$
R_{\lambda, 0} = \left( \lambda + L(A,0) \right)^{-1} \in \mathcal{L} \left( \mathrm{L}^{2}(\Omega) \right).
$$
Hence $\rho(R_{\lambda, 0}) \geq \lambda^{-1}$ is implied.
\item Furthermore $0 < \lambda^{-1} \leq \rho(R_{\lambda, 0}) \leq \rho(R_{\lambda})$ is true.\\
\end{enumerate}
\end{lem}

\begin{proof}
\begin{enumerate}[i)]
\item For every $v \in H^{1}(\Omega)$ we infer
$$
a(A, 0)({\bf 1}, v) = \int_{\Omega} A \nabla {\bf 1} \cdot \overline{\nabla v} \d x = 0 =\langle {\bf 0}, v \rangle_{\mathrm{L}^{2}(\Omega)}.
$$
Hence ${\bf 1} \in D\left( L(A, 0) \right)$ and $L(A, 0) {\bf 1} = 0$ is true almost everywhere in $\Omega$. We conclude
that 
$$
\left( \lambda + L(A,0) \right) {\bf 1} = \lambda {\bf 1}
$$ 
is true which implies $R_{\lambda, 0}  {\bf 1} = \lambda^{-1}  {\bf 1}$.\\
\item In Lemma \ref{resolvent_comparison} we have already seen $0 \leq R_{\lambda, 0} \leq R_{\lambda}$. 
For any $n \in \N$ we infer that
$$
0 \leq R_{\lambda, 0}^{n} \leq R_{\lambda}^{n}
$$ 
is true and therefore
$$
0 < \lambda^{-1} \leq \rho(R_{\lambda, 0}) = \lim_{n \to \infty} \| R_{\lambda, 0}^{n} \|^{1/n} \leq 
\lim_{n \to \infty} \| R_{\lambda}^{n} \|^{1/n} = \rho(R_{\lambda}) 
$$
proves the claim.
\end{enumerate}
\end{proof}

\begin{thm}\label{mu_eigenvalue}
There exists a positive eigenfunction $0 \leq \phi \in D\left( L(A, -C) \right)$ of  $L(A, -C)$ to an eigenvalue $\mu \in (-\infty,0]$, i.e.
$$
L(A, -C) \phi = \mu \phi.
$$
\end{thm}

\begin{proof}
The resolvent $R_{\lambda}$ is a positive and compact operator with $0 < \rho(R_{\lambda})$. Hence by Krein-Rutman Theorem
the spectral radius $\rho(R_{\lambda})$ is an eigenvalue of $R_{\lambda}$ with an eigenfunction 
$0 \neq \phi \in \mathrm{L}^{2}(\Omega, \R)$. Hence
$$
\phi = \rho(R_{\lambda})^{-1} R_{\lambda}\phi \geq 0
$$
holds almost everywhere in $\Omega$. Furthermore $\phi$ is contained in $D\left( L(A, -C) \right)$ with
$$
L(A, -C) \phi = \mu \phi
$$
is true for $\mu = \rho(R_{\lambda})^{-1} - \lambda \leq 0$ by Lemma \ref{1_Neumann_eigenvalue}.\\
\end{proof}

\subsection{Upper bound of $\phi$}

Our aim is to prove that the positive eigenfunction $\phi$ obtained in
Subsection \ref{existence_phi} belongs to $L^\infty(\Omega)$. We derive this 
conclusion directly from the weak eigenvalue equation
$$
    L(A,-C)\phi=\mu\phi
$$
without using any additional elliptic regularity.\\
\ \\
The proof is based on an iteration of integrability exponents. We first
describe the main steps of the argument. The details are given in
Subsections \ref{Power_truncation_subsection} and \ref{iteration_subsection}. 
Throughout this subsection, we set
$$
    \kappa=\frac{d-1}{d-2}>1.
$$
\begin{enumerate}
\item For $p \geq 2$ and the choice of $\kappa$ the Sobolev embedding theorem yields  a constant $S>0$ satisfying
$$
\| \phi^{p/2} \|_{\mathrm{L}^{2\kappa}(\Omega)}^{2} \leq 
S \left( \left\| \nabla ( \phi^{p/2} ) \right\|_{\mathrm{L}^{2}(\Omega)^{d}}^{2} + 
\left\| \phi^{p/2} \right\|_{\mathrm{L}^{2}(\Omega)}^{2} \right).
$$
Using $\| \phi^{p/2} \|_{\mathrm{L}^{2\kappa}(\Omega)}^{2} = \| \phi \|_{\mathrm{L}^{p \kappa}(\Omega)}^{p}$ and
$\left\| \phi^{p/2} \right\|_{\mathrm{L}^{2}(\Omega)}^{2} = \left\| \phi \right\|_{\mathrm{L}^{p}(\Omega)}^{p}$ we obtain 
\begin{equation}
\| \phi \|_{\mathrm{L}^{p \kappa}(\Omega)}^{p} \leq 
S \left( \left\| \nabla ( \phi^{p/2} ) \right\|_{\mathrm{L}^{2}(\Omega)^{d}}^{2} + \left\| \phi \right\|_{\mathrm{L}^{p}(\Omega)}^{p} \right).
\end{equation}

\item Then Lemma \ref{truncation_lemma} gives 
$$
\left\| \nabla ( \phi^{p/2} ) \right\|_{\mathrm{L}^{2}(\Omega)^{d}}^{2} \ \leq \ Kp^{2} 
\left\| \phi \right\|_{\mathrm{L}^{p}(\Omega)}^{p}
$$
for a constant $K \geq 1$ and $p \geq 2$ which implies
$$
\| \phi \|_{\mathrm{L}^{p \kappa}(\Omega)}^{p} \leq \tilde{K}p^{2} \left\| \phi \right\|_{\mathrm{L}^{p}(\Omega)}^{p}.
$$

\item In Theorem \ref{iteration_thm} we apply this estimate repeatedly to prove $\phi \in \mathrm{L}^{\infty}(\Omega)$
by an iteration. We start by choosing $p_{0} = 2$ and define $p_{n+1}=\kappa p_{n}$. Then
$$
\| \phi \|_{\mathrm{L}^{p_{n}}(\Omega)} \ \leq \ \left( \tilde{K} p_{n-1}^{2} \right)^{1/p_{n-1}} 
\| \phi \|_{\mathrm{L}^{p_{n-1}}(\Omega)} \ = \ \| \phi \|_{\mathrm{L}^{2}(\Omega)} 
\prod_{j=0}^{n-1} \left( \tilde{K} p_{j}^{2} \right)^{1/p_{j}}
$$
is implied. Using the convergence of an infinite product on the right hand side gives
$$
\| \phi \|_{\mathrm{L}^{p_{n}}(\Omega)} \ \leq \ M \| \phi \|_{\mathrm{L}^{2}(\Omega)} 
$$
for a constant $M > 0$ and any $n \in \N$.

\item Let $a > M  \| \phi \|_{\mathrm{L}^{2}(\Omega)}$ be arbitrary and define $A = \left\{  x \in \Omega \ : \ \phi(x) > a \right\}$.
Then
$$
\left| A \right| = \int_{A} {\bf 1} \d x \leq a^{-p_{n}} \int_{A} \phi^{p_{n}} \d x 
\leq  a^{-p_{n}}  \| \phi \|_{\mathrm{L}^{p_{n}}(\Omega)}^{p_{n}}
\leq \left( \frac{M \| \phi \|_{\mathrm{L}^{2}(\Omega)}}{a} \right)^{p_{n}} \to 0
$$
for $n \to \infty$ implies $\left| A \right| = 0$. Hence $\phi \in L^\infty(\Omega)$ with
$$
0 \ \leq \ \phi \leq M  \| \phi \|_{\mathrm{L}^{2}(\Omega)}
$$
almost everywhere in $\Omega$.
\end{enumerate}

\subsubsection{An $H^1$-estimate for powers of $\phi$} \label{Power_truncation_subsection}

We first prove the $H^1$-estimate for powers of $\phi$ which is needed
in the iteration. The formal choice of $\phi^{p-1}$ as a test function
in the weak eigenvalue equation is not justified at this point.
Therefore, we approximate this function by power truncations.
We start our argumentation with a rather technical but useful proposition.\\

\begin{prp}\label{truncations_phi}
Let $n\in\mathbb N$, $p \geq 2$.  For $t > 0$, define
$$
F_{n}(t)=t(t\wedge n)^{p-2}, \quad G_{n}(t)=t(t\wedge n)^{(p-2)/2} 
$$
and set $F_{n}(0) = 0 = G_{n}(0)$. Then $F_{n} \circ \phi$ and $G_{n} \circ \phi$ are both contained in 
$H^{1}(\Omega)$ with
\begin{enumerate}[a)]
\item $\nabla \left( F_{n}( \phi ) \right) = F_{n}^{\prime}(\phi) \nabla \phi$ and
\item $\nabla \left( G_{n}( \phi ) \right) = G_{n}^{\prime}(\phi) \nabla \phi$
\end{enumerate}
almost everywhere in $\Omega$.
\end{prp}

\begin{proof}
We will only prove the claim for $F_{n}( \phi )$ since the same arguments apply to  $G_{n}( \phi )$. 
Our argumentation is based on Theorem 2.1.11 on page 48 in \cite{Ziemer89} which states that
$$
\psi \circ u \in H^{1}(\Omega) \ \text{with} \ \partial_{j} (\psi \circ u) = \psi^{\prime}(u) \partial_{j} u
$$
holds almost everywhere in $\Omega$ for $\psi \colon \R \to \R$ being Lipschitz continuous and
$u \in H^{1}(\Omega)$ being real-valued.
\begin{enumerate}[i)]
\item We show that $F_{n}$ is Lipschitz continuous in $[0,\infty)$. For $t > n$ the derivative of $F_{n}$
at $t$ is given by $F_{n}^{\prime}(t) = n^{p - 2}$. For $0<t<n$ it is given by 
$$
F_{n}^{\prime}(t) = (p-1) t^{p - 2} \leq (p-1) n^{p - 2}. 
$$
Using the mean value theorem for $0 \leq s \leq t \leq n$ we infer that
$$
\left| F_{n}(t) - F_{n}(s) \right| \ \leq \ (p-1) n^{p - 2} |t-s|
$$
is true. Similarly we show Lipschitz continuity on $[n, \infty)$. Now let 
$$
0 \leq s \leq n \leq t < \infty
$$ 
with $s \neq t$ then
\begin{align*}
\left| F_{n}(t) - F_{n}(s) \right| & \leq  \left| F_{n}(t) - F_{n}(n) \right| + \left| F_{n}(n) - F_{n}(s) \right| \\
& \leq (p-1) n^{p - 2} (t-n) + (p-1) n^{p - 2} (n-s) = (p-1) n^{p - 2} (t-s)  
\end{align*}
is true and therefore $F_{n}$ is Lipschitz continuous on $[0,\infty)$.
\item Since $F_{n}(0) = 0$, we can extend $F_{n}$ continuously to the entire set $\R$ with $0$ on $(-\infty,0)$ 
without losing Lipschitz continuity.
\item The claim is a direct implication of the cited result above since $\phi \in H^{1}(\Omega)$ is positive.
\end{enumerate}
\end{proof}

Let us now focus on the central argument of this subsection. Notice that the formal choice of $v=\phi^{p-1}$
in the weak eigenvalue equation is not justified before one knows that the relevant power of $\phi$ is contained
in $H^{1}(\Omega)$. The following truncation argument supplies this missing step.\\

\begin{lem}\label{truncation_lemma}
Let $p \geq 2$ and $0 \leq \phi \in D( L(A,-C) )$ satisfy $L(A,-C)\phi = \mu \phi$. Assume in addition that  
\begin{enumerate}[a)]
\item $\phi \in \mathrm{L}^{p}(\Omega)$ and
\item $\gamma(\phi) \in \mathrm{L}^{p}(\partial \Omega)$
\end{enumerate}
are satisfied. Then $\phi^{p/2} \in H^{1}(\Omega)$. Moreover, there exists a constant $K \geq 1$, independent
of $p$, such that
\begin{equation}\label{power_truncation}
\left\| \nabla ( \phi^{p/2} ) \right\|_{\mathrm{L}^{2}(\Omega)^{d}}^{2} \ \leq \ Kp^{2} 
\left\| \phi \right\|_{\mathrm{L}^{p}(\Omega)}^{p}.
\end{equation}
\end{lem}

\begin{proof}
For $n \in \N$, let $F_{n}$ and $G_{n}$ be the power truncations of Proposition \ref{truncations_phi}.
We define 
\begin{enumerate}[1)]
\item $v_{n} = F_{n} \circ \phi = \phi \phi_{n}^{p-2} \in H^{1}(\Omega)$ and  
\item $w_{n} = G_{n} \circ \phi = \phi \phi_{n}^{(p-2)/2} \in H^{1}(\Omega)$ 
\end{enumerate}
with $\phi_{n} = \min\{\phi, n\}$. The following proof is divided into several smaller steps for a better understanding.

\begin{enumerate}[i)]
\item Notice that $w_{n}$ converges pointwise to $\phi^{p/2}$ almost everywhere and 
$$
0 \ \leq \ w_{n}^{2} \ \leq \ \phi^{p}
$$
holds almost everywhere in $\Omega$ with $\phi^{p} \in \mathrm{L}^{1}(\Omega)$. Hence the dominated convergence 
theorem gives a strong convergence of 
$(w_{n})_{n \in \N}$ to $\phi^{p/2}$ in $\mathrm{L}^{2}(\Omega)$.\\

\item In the following argumentation we prove that $(w_{n})_{n \in \N}$ is a bounded sequence in $H^{1}(\Omega)$. 
In that case, there exists a subsequence of $(w_{n})_{n \in \N}$, which we do not relabel, and a function $w \in H^{1}(\Omega)$ 
such that
$$
w_{n} \ \rightharpoonup \ w  
$$
in $H^{1}(\Omega)$. Since $w_{n} \to \phi^{p/2}$ strongly in $\mathrm{L}^{2}(\Omega)$, it follows
$$
\phi^{p/2} = w \in H^{1}(\Omega).
$$

\item Set 
$$
c_{p} = \frac{4(p-1)}{p^{2}} \in (0,1].
$$
For every $t > 0$ with $t \neq n$ we conclude that
$$
F_{n}^{\prime}(t) \ \geq \ c_{p} \left( G_{n}^{\prime}(t) \right)^{2}
$$
is true. Notice that $v_{n}$ are contained in $H^{1}(\Omega)$ and Proposition \ref{truncations_phi}
implies
\begin{align*}
\Re \int_{\Omega} A \nabla \phi \cdot \nabla v_{n} \d x & \geq \alpha  \int_{\Omega} F_{n}^{\prime}(\phi) | \nabla \phi |^{2} \d x \\
& \geq \alpha c_{p} \int_{\Omega} \left( G_{n}^{\prime}(\phi) \right)^{2} | \nabla \phi |^{2} \d x \\
& = \alpha c_{p} \int_{\Omega}  | \nabla w_{n} |^{2} \d x 
\end{align*}
The weak eigenvalue equation is given by
\begin{equation}
\int_{\Omega} A \nabla \phi \cdot \overline{\nabla v} \d x - \langle C \gamma(\phi), \gamma(v) \rangle_{\mathrm{L}^{2}(\partial \Omega)}
= \langle L(A,-C)\phi, v \rangle_{\mathrm{L}^{2}(\Omega)} = \mu \int_{\Omega} \phi \overline{v} \d x
\end{equation}
for every $v \in H^{1}(\Omega)$. Using $v_{n}$ as a test function results in
\begin{equation}\label{weak_eigenvalue_result_v_n}
\alpha c_{p}  \| \nabla w_{n} \|_{\mathrm{L}^{2}(\Omega)^{d}}^{2} \leq |\mu| \int_{\Omega} \phi^{2} \phi_{n}^{p-2} \d x
+ \left| \langle C \gamma(\phi), \gamma(v_{n}) \rangle_{\mathrm{L}^{2}(\partial \Omega)} \right|.
\end{equation}

\item Let us focus on the term $\left| \langle C \gamma(\phi), \gamma(v_{n}) \rangle_{\mathrm{L}^{2}(\partial \Omega)} \right|$ next. 
We use Theorem A.2 in \cite{MaggiVillani2005} to argue that
$$
\gamma(v_{n}) = \gamma (F_{n}(\phi)) = F_{n}(\gamma(\phi)) = \gamma(\phi) (\gamma(\phi) \wedge n)^{p-2}
$$
is true. Define $h = \gamma(\phi)$ for further use simply for aesthetic reasons. 
Notice that 
$$
0 \leq h(h \wedge n)^{p-2} \leq h^{p-1}
$$ 
holds pointwise and for $p' = p/(p-1)$ we conclude to
$$
\| h(h \wedge n)^{p-2} \|_{\mathrm{L}^{p'}(\partial \Omega)} \leq \| h^{p-1} \|_{\mathrm{L}^{p'}(\partial \Omega)} 
= \| h \|_{\mathrm{L}^{p}(\partial \Omega)}^{p-1}.
$$
By Lemma \ref{C_interpolation} there exists a constant $M_{C} > 0$ such that
$$
\left\| C \right\|_{\mathrm{L}^{p} \to \mathrm{L}^{p}} \ \leq \ M_{C}
$$
is true for every $p \in [2,\infty]$. Hence we use Hölder's inequality to conclude to
$$
\left| \langle C \gamma(\phi), \gamma(v_{n}) \rangle_{\mathrm{L}^{2}(\partial \Omega)} \right| \ \leq \ M_{C} 
\| h \|_{\mathrm{L}^{p}(\partial \Omega)}^{p}.
$$

\item Also $ \phi^{2} \phi_{n}^{p-2} \leq \phi^{p}$ pointwise almost everywhere in $\Omega$ and therefore
(\ref{weak_eigenvalue_result_v_n}) implies
\begin{equation}\label{weak_eigenvalue_result_w_n}
\alpha c_{p}  \| \nabla w_{n} \|_{\mathrm{L}^{2}(\Omega)^{d}}^{2} \leq |\mu| \| \phi \|_{\mathrm{L}^{p}( \Omega)}^{p}
+ M_{C} \| h \|_{\mathrm{L}^{p}(\partial \Omega)}^{p}.
\end{equation}
Also 
$$
\| w_{n} \|_{\mathrm{L}^{2}( \Omega)}^{2} = \int_{\Omega} \phi^{2} \phi_{n}^{p-2} \d x \leq 
\| \phi \|_{\mathrm{L}^{p}( \Omega)}^{p}
$$
is true. Hence $(w_{n})_{n \in \N}$ is a bounded sequence in $H^{1}(\Omega)$. By i) and ii) we conclude that 
$\phi^{p/2}$ is contained in $H^{1}(\Omega)$. Furthermore, $w_{n}$ converges weakly to $\phi^{p/2}$ in $H^{1}(\Omega)$
but strongly in $\mathrm{L}^{2}( \Omega)$. This implies a weak convergence of $\nabla w_{n}$ to $\nabla \phi^{p/2}$
in $\mathrm{L}^2(\Omega)^{d}$. Proposition 3.5 (iii) on page 58 in \cite{Brezis11} states:
$$
\| \nabla \phi^{p/2} \|_{\mathrm{L}^2(\Omega)^{d}}^{2} \ \leq \ \liminf_{n \to \infty}  \| \nabla w_{n} \|_{\mathrm{L}^2(\Omega)^{d}}^{2}.
$$
Passing to the limit in $(\ref{weak_eigenvalue_result_w_n})$, we obtain 
$$
\alpha c_{p} \| \nabla \phi^{p/2} \|_{\mathrm{L}^2(\Omega)^{d}}^{2} \ \leq \ |\mu| \| \phi^{p/2} \|_{\mathrm{L}^{2}( \Omega)}^{2}
+ M_{C} \| \gamma(\phi)^{p/2} \|_{\mathrm{L}^{2}(\partial \Omega)}^{2}
$$
where we remember $h =  \gamma(\phi)$.\\

\item We show $h^{p/2} = \gamma(\phi^{p/2})$ next. Along the subsequence chosen in ii), $w_{n}$ converges weakly to $\phi^{p/2}$ in 
$H^{1}(\Omega)$. We use the adjoint trace operator $\gamma^{\ast} \colon \mathrm{L}^{2}(\partial \Omega) \to H^{1}(\Omega)$
with 
$$
\langle \gamma(w_{n}), g \rangle_{\mathrm{L}^{2}(\partial \Omega)} = \langle w_{n}, \gamma^{\ast}(g) \rangle_{H^{1}(\Omega)}
$$
for any $g \in \mathrm{L}^{2}(\partial \Omega)$, to argue that $\gamma(w_{n})$ converges weakly to $\gamma(\phi^{p/2})$
in $\mathrm{L}^{2}(\partial \Omega)$. \\
On the other hand, we use Theorem A.2 in \cite{MaggiVillani2005} once again, similar to our argumentation referring 
$\gamma(v_{n})$, to argue that
$$
\gamma(w_{n}) = h(h \wedge n)^{(p-2)/2}
$$
is true. Then $h(h \wedge n)^{(p-2)/2} \to h^{p/2}$ pointwise on $\partial \Omega$ for $n \to \infty$ and
$$
\gamma(w_{n})^{2} = h^{2}(h \wedge n)^{(p-2)} \leq h^{p} \in \mathrm{L}^{1}(\partial \Omega).
$$ 
Once again we use the dominated convergence theorem to conclude that $\gamma(w_{n})$ converges strongly to $h^{p/2}$
in $\mathrm{L}^{2}(\partial \Omega)$ which finally implies 
$$
\gamma(\phi^{p/2}) = h^{p/2}
$$
by the uniqueness of a weak limit.\\

\item Combining parts v) and vi), we obtain
\begin{equation}\label{almost_claim}
\alpha c_{p} \| \nabla \phi^{p/2} \|_{\mathrm{L}^2(\Omega)^{d}}^{2} \ \leq \ |\mu| \| \phi^{p/2} \|_{\mathrm{L}^{2}( \Omega)}^{2}
+ M_{C} \| \gamma(\phi^{p/2}) \|_{\mathrm{L}^{2}(\partial \Omega)}^{2}.
\end{equation}
By Lemma \ref{gamma} there exists a constant $k > 0$ such that
$$
\| \gamma(\phi^{p/2}) \|_{\mathrm{L}^{2}( \partial \Omega)}^{2} \ \leq \ \varepsilon  \| \nabla \phi^{p/2} \|_{\mathrm{L}^2(\Omega)^{d}}^{2}
+ k (1+\varepsilon^{-1}) \| \phi^{p/2} \|_{\mathrm{L}^{2}(\Omega)}^{2}
$$ 
is true for every $\varepsilon > 0$. Choosing $\varepsilon = \alpha c_{p}/(2M_{C})$ we rewrite (\ref{almost_claim}) to
$$
 \| \nabla \phi^{p/2} \|_{\mathrm{L}^2(\Omega)^{d}}^{2} \ \leq \ K_{0} (c_{p}^{-1} + c_{p}^{-2}) 
\| \phi^{p/2} \|_{\mathrm{L}^{2}(\Omega)}^{2} \ = \ K_{0} (c_{p}^{-1} + c_{p}^{-2}) 
\| \phi \|_{\mathrm{L}^{p}(\Omega)}^{p}
$$
for a constant $K_{0}$ independent of $p$. Finally, for $p \geq 2$ we conclude that $c_{p}^{-1} \leq p/2$ and therefore
$$
\| \nabla \phi^{p/2} \|_{\mathrm{L}^2(\Omega)^{d}}^{2} \ \leq \ \frac{3}{4} K_{0} p^{2} \| \phi \|_{\mathrm{L}^{p}(\Omega)}^{p}
$$
is true which proves the claim.
\end{enumerate}
\end{proof}

\begin{cor}\label{trace_iteration}
Let $p \geq 2$ and $0 \leq \phi \in D( L(A,-C) )$ satisfy  
\begin{enumerate}[a)]
\item $L(A,-C)\phi = \mu \phi$,
\item $\phi \in \mathrm{L}^{p}(\Omega)$ and
\item $\gamma(\phi) \in \mathrm{L}^{p}(\partial \Omega)$.
\end{enumerate}
Then $\gamma(\phi)$ is contained in $\mathrm{L}^{\kappa p}(\partial \Omega)$.
\end{cor}

\begin{proof}
We apply Theorem A.3 in \cite{MaggiVillani2005} for $q=2$. Since
$q^{\sharp} = \frac{2(d-1)}{d-2} = 2 \kappa$ we obtain
$$
\| \gamma(\phi^{p/2}) \|_{\mathrm{L}^{2 \kappa}(\partial \Omega)} \ \leq \ C_{\Omega} 
\left( \| \phi^{p/2} \|_{\mathrm{L}^{1}( \Omega)} +   \| \nabla (\phi^{p/2}) \|_{\mathrm{L}^{2}( \Omega)^{d}}  \right)
$$ 
for a constant $C_{\Omega} > 0$. Furthermore $\| \phi^{p/2} \|_{\mathrm{L}^{1}( \Omega)} \leq |\Omega|^{1/2} 
\| \phi^{p/2} \|_{\mathrm{L}^{2}( \Omega)}$ is true due to the boundedness of $\Omega$. We use Lemma \ref{truncation_lemma}
to conclude to
\begin{align*}
\| \gamma(\phi^{p/2}) \|_{\mathrm{L}^{2 \kappa}(\partial \Omega)}^{2} & \leq S_{\Omega} 
\left( \| \phi^{p/2} \|_{\mathrm{L}^{2}( \Omega)}^{2} +   \| \nabla (\phi^{p/2}) \|_{\mathrm{L}^{2}( \Omega)^{d}}^{2}  \right) \\
& \leq S_{\Omega} (Kp^{2} + 1) \| \phi \|_{\mathrm{L}^{p}(\Omega)}^{p}
\end{align*}
holds for a constant $S_{\Omega} > 0$. Finally using $\gamma(\phi^{p/2}) = \gamma(\phi)^{p/2}$ from vi) in the proof of 
Lemma \ref{truncation_lemma} gives $\gamma(\phi)^{p/2} \in \mathrm{L}^{2 \kappa}(\partial \Omega)$. So
$\gamma(\phi)$ is contained in $\mathrm{L}^{\kappa p}(\partial \Omega)$.\\
\end{proof}

\subsubsection{Iteration to $\mathrm{L}^{\infty}(\Omega)$}\label{iteration_subsection}

We now use a constant $\kappa = (d-1)/(d-2) > 1$ for all following arguments. 

\begin{lem}\label{iteration}
Let $p \geq 2$ and $0 \leq \phi \in D( L(A,-C) )$ satisfy  
\begin{enumerate}[a)]
\item $L(A,-C)\phi = \mu \phi$,
\item $\phi \in \mathrm{L}^{p}(\Omega)$ and
\item $\gamma(\phi) \in \mathrm{L}^{p}(\partial \Omega)$.
\end{enumerate}
Then $\phi \in \mathrm{L}^{\kappa p}(\Omega)$ is true and further there exists a constant $\tilde{K}> 0$ independent of $p$ 
with 
$$
\| \phi \|_{\mathrm{L}^{\kappa p}(\Omega)}^{p} \ \leq \ \tilde{K} p^{2} 
\| \phi \|_{\mathrm{L}^{p}(\Omega)}^{p}.
$$
\end{lem}

\begin{proof}
By Lemma \ref{truncation_lemma} we conclude that $\phi^{p/2} \in H^{1}(\Omega)$ is implied with
$$
\left\| \nabla ( \phi^{p/2} ) \right\|_{\mathrm{L}^{2}(\Omega)^{d}}^{2} \ \leq \ Kp^{2} 
\left\| \phi \right\|_{\mathrm{L}^{p}(\Omega)}^{p}.
$$
for a constant $K>0$. We use 
Sobolev embedding theorems for the choice of $\kappa$ to argue that there exists a constant $S>0$ satisfying 
\begin{align*}
\| \phi^{p/2} \|_{\mathrm{L}^{2\kappa}(\Omega)}^{2} & \leq S \| \phi^{p/2} \|_{H^{1}(\Omega)}^{2}
= S \left( \left\| \nabla ( \phi^{p/2} ) \right\|_{\mathrm{L}^{2}(\Omega)^{d}}^{2} + 
\left\| \phi^{p/2} \right\|_{\mathrm{L}^{2}(\Omega)}^{2} \right)  \\
& \leq S \left( Kp^{2} 
\left\| \phi \right\|_{\mathrm{L}^{p}(\Omega)}^{p} + \left\| \phi \right\|_{\mathrm{L}^{p}(\Omega)}^{p}  \right) 
= S(Kp^{2} + 1)  \left\| \phi \right\|_{\mathrm{L}^{p}(\Omega)}^{p}  
\end{align*}
which proves the claim.\\
\end{proof}

Now we prove an essential boundedness of the positive eigenfunction $\phi$ of $L(A,-C)$ by an iteration argument.

\begin{thm}\label{iteration_thm}
Let $\phi \in D( L(A,-C) )$ be the positive eigenfunction of  $L(A,-C)$ from Subsection \ref{existence_phi} with no 
additional assumptions apart from $L(A,-C)\phi = \mu \phi$. Then $\phi \in \mathrm{L}^{\infty}(\Omega)$ is implied.
\end{thm}

\begin{proof}
Set $p_{0} = 2$ and $p_{n+1} = \kappa p_{n}$. Notice that $p_{n} = 2\kappa^{n} \to \infty$
for $n \to \infty$ since $\kappa > 1$. 
\begin{enumerate}[i)]
\item Initially $\phi$ is contained in $\mathrm{L}^{p_{0}}(\Omega)$ with 
$\gamma(\phi) \in \mathrm{L}^{p_{0}}(\partial \Omega)$.
By Lemma \ref{iteration} and Corollary \ref{trace_iteration} we conclude that $\phi$ is contained in 
$\mathrm{L}^{p_{1}}(\Omega)$ with $\gamma(\phi) \in \mathrm{L}^{p_{1}}(\partial \Omega)$. \\

\item Iteratively $\phi$ and $\gamma(\phi)$ are contained in 
$\mathrm{L}^{p_{n}}(\Omega)$ with $\mathrm{L}^{p_{n}}(\partial \Omega)$ respectively for any $n \in \N$. 
Furthermore Lemma \ref{iteration} states that
$$
\| \phi \|_{\mathrm{L}^{p_{n}}(\Omega)} \ \leq \ \left( \tilde{K} p_{n-1}^{2} \right)^{1/p_{n-1}} 
\| \phi \|_{\mathrm{L}^{p_{n-1}}(\Omega)} \ = \ \| \phi \|_{\mathrm{L}^{2}(\Omega)} 
\prod_{j=0}^{n-1} \left( \tilde{K} p_{j}^{2} \right)^{1/p_{j}}.
$$

\item Notice that $\prod_{j=0}^{n-1} \left( \tilde{K} p_{j}^{2} \right)^{1/p_{j}}$ converges for $n \to \infty$ since 
$$
\sum_{j=0}^{\infty} \frac{\ln(\tilde{K}) + 2 \ln(p_{j}) }{p_{j}} = \frac{1}{2} \left( \ln(\tilde{K}) + 2 \ln(2) \right) 
\sum_{j=0}^{\infty} \kappa^{-j} + \ln (\kappa) \sum_{j=0}^{\infty} j \kappa^{-j} < \infty
$$
is true where we used $p_{j} = 2 \kappa^{j}$ and the ratio test for the convergence of a series to justify
the convergence of $\sum_{j=0}^{\infty} j \kappa^{-j}$. 
Therefore there exists a constant $M > 0$ independent of $n \in \N$ such that
$$
\| \phi \|_{\mathrm{L}^{p_{n}}(\Omega)} \ \leq \ M  \| \phi \|_{\mathrm{L}^{2}(\Omega)}.
$$

\item Now choose $a > M  \| \phi \|_{\mathrm{L}^{2}(\Omega)}$ and define a set $A$ by
$A = \left\{  x \in \Omega \ : \ \phi(x) > a \right\}$.
Then
$$
\left| A \right| = \int_{A} {\bf 1} \d x \leq a^{-p_{n}} \int_{A} \phi^{p_{n}} \d x 
\leq  a^{-p_{n}}  \| \phi \|_{\mathrm{L}^{p_{n}}(\Omega)}^{p_{n}}
\leq \left( \frac{M \| \phi \|_{\mathrm{L}^{2}(\Omega)}}{a} \right)^{p_{n}} \to 0
$$
is true for $n \to \infty$. Hence $A$ is a Lebesgue null set and therefore 
$$
0 \leq \phi \leq M  \| \phi \|_{\mathrm{L}^{2}(\Omega)}
$$
is true almost everywhere in $\Omega$.
\end{enumerate}
\end{proof}

\subsection{Lower bound of $\phi$}

The aim of our argumentation is to prove that the positive eigenfunction $\phi$ obtained in the preceding 
Subsection \ref{existence_phi} is bounded below by a strictly positive constant $\delta$.
We introduce the mean-value projection $P \colon \mathrm{L}^{2}(\Omega) \to \mathrm{L}^{2}(\Omega)$ defined as
\begin{equation}\label{definition_P}
Pu \ = \ \frac{1}{|\Omega|} \langle u, {\bf 1} \rangle_{\mathrm{L}^{2}(\Omega)} {\bf 1} \ = \ 
\frac{1}{|\Omega|} \left( \int_{\Omega} u \d x \right) {\bf 1}
\end{equation}
and we use the Neumann semigroup $( \mathrm{e}^{-tL(A,0)} )_{t \geq 0}$ in 
$\mathrm{L}^{2}(\Omega)$ which was introduced in Subsection \ref{Neumann_semigroups_introduction}.
Our main arguments are:
\begin{enumerate}
\item A semigroup domination
\begin{equation}\label{Neumann_comparison_inequ}
0 \ \leq \ \mathrm{e}^{-tL(A,0)} \ \leq \ \mathrm{e}^{-tL(A,-C)}
\end{equation}
which is implied by Lemma \ref{resolvent_comparison} in combination with Euler's formula. Please see Theorem 2.4.2 in \cite{Arendt06}.
\item The asymptotic behaviour of $( \mathrm{e}^{-tL(A,0)} )_{t \geq 0}$ in 
$\mathcal{L} \left(\mathrm{L}^{2}(\Omega) \right)$. In particular,
we use Theorem 10.4.1 in \cite{Arendt06} to conclude 
$$
\left\| \mathrm{e}^{-tL(A,0)}  - P \right\|_{\mathrm{L}^{2} \to \mathrm{L}^{2}} \ \to \ 0 \ \text{for} \ t \to \infty.
$$
\item An ultracontractivity of the Neumann semigroup $( \mathrm{e}^{-tL(A,0)} )_{t \geq 0}$ in 
$\mathrm{L}^{2}(\Omega)$.\\
\end{enumerate}
First, we focus on the asymptotic behaviour of $( \mathrm{e}^{-tL(A,0)} )_{t \geq 0}$ in 
$\mathcal{L} \left(\mathrm{L}^{2}(\Omega) \right)$.\\

\begin{lem}\label{Neumann_asymptotic}
The Neumann semigroup converges in operator norm to the
mean-value projection $P$, that is,
$$
\left\| \mathrm{e}^{-tL(A,0)}  - P \right\|_{\mathrm{L}^{2} \to \mathrm{L}^{2}} \ \to \ 0 \ (t \to \infty).
$$
\end{lem}

\begin{proof}
We verify the assumptions of Theorem 10.4.1 in \cite{Arendt06}.
\begin{enumerate}[i)]
\item The Neumann semigroup is positive and irreducible by Theorem 11.2.1 in \cite{Arendt06} 
since $\Omega$ is connected and the form domain is $H^1(\Omega)$.

\item The shifted form
$$
b(A,0)(u,v) = a(A,0)(u,v) +\langle u,v \rangle_{\mathrm{L}^2(\Omega)}
$$
is coercive. Hence $-(L(A,0)+I)$ generates a holomorphic semigroup
by Theorem 7.1.5 in \cite{Arendt06} . Consequently,
$-L(A,0)$ also generates a holomorphic semigroup.

\item The resolvents of $L(A,0)$ are compact since
$H^{1}(\Omega)$ is compactly embedded into $\mathrm{L}^{2}(\Omega)$.

\item Also $L(A,0) {\bf 1} = 0$ as well as 
$L(A,0)^{\ast} {\bf 1}  = L(A^{T},0) {\bf 1}  = 0$ are both true. \\
\end{enumerate}
It remains to determine the spectral bound of the generator. 
Compactness of the resolvents 
implies that every spectral value of $L(A,0)$ is an eigenvalue. If $L(A,0)u = \lambda u$ and $u \neq 0$ then accretivity gives
$$
\Re \lambda \|u\|_{\mathrm{L}^{2}(\Omega)}^{2} = \Re a(A,0)(u,u) \geq 0.
$$
Thus $\inf \{ \Re \lambda \ : \ \lambda \in \sigma(L(A,0)) \} \ \geq \ 0$. Since $L(A,0){\bf 1} = 0$, equality holds. Consequently,
\begin{align*}
s(-L(A,0)) & = \sup \{ \Re \mu \ : \ \mu \in \sigma(-L(A,0)) \} \\
& = \sup \{ -\Re \lambda \ : \ \lambda \in \sigma(L(A,0)) \} \\
& = - \inf \{ \Re \lambda \ : \ \lambda \in \sigma(L(A,0)) \} = 0.
\end{align*}
All necessary assumptions are shown. The claim is implied by Theorem 10.4.1 in \cite{Arendt06}.\\ 
\end{proof}

Now we need to establish the next important result of our argumentation. We prove the ultracontractivity 
of $( \mathrm{e}^{-tL(A,0)} )_{t \geq 0}$ in $\mathrm{L}^{2}(\Omega)$.\\

\begin{lem}\label{Neumann_ultracontractivity}
For every $t > 0$ the operator
$$
\mathrm{e}^{-tL(A,0)} \ \colon \ \mathrm{L}^{2}(\Omega) \to \mathrm{L}^{\infty}(\Omega)
$$
is bounded. More precisely, there exists $c > 0$ such that
$$
\left\| \mathrm{e}^{-tL(A,0)} \right\|_{\mathrm{L}^{2} \to \mathrm{L}^{\infty}} \ \leq \ c\mathrm{e}^{t}t^{-d/4}
$$
is true for every $t > 0$.
\end{lem}

\begin{proof}
The shifted form
$$
b(A,0)(u,v) = a(A,0)(u,v) + \langle u, v \rangle_{\mathrm{L}^{2}(\Omega)}
$$
with $u,v \in D(b(A,0)) = H^{1}(\Omega)$ is coercive and its associated semigroup is 
$$
S(t) = \mathrm{e}^{-t}\mathrm{e}^{-tL(A,0)}.
$$
Since $\Omega$ is a bounded Lipschitz domain, $H^{1}(\Omega)$ satisfies the Nash inequality of dimension $d$.
Moreover, the semigroup associated with $b(A,0)$ and its adjoint are submarkovian. Theorem 12.3.2 in \cite{Arendt06}
gives
$$
\left\| S(t) \right\|_{\mathrm{L}^{2} \to \mathrm{L}^{\infty}} \ \leq \ c t^{-d/4}
$$
which implies the claim since $S(t) = \mathrm{e}^{-t}\mathrm{e}^{-tL(A,0)}$.\\
\end{proof}

\begin{rem}
Notice that the shift in the previous proof is not introduced to obtain submarkovianity: both $\mathrm{e}^{-tL(A,0)}$ 
and its adjoint are already submarkovian as needed for Theorem 12.3.2 in \cite{Arendt06}.
Its purpose is to absorb the $\mathrm{L}^2$-term occurring in the inhomogeneous Nash inequality 
into the form $b(A,0)$. Indeed, a homogeneous Nash inequality cannot
hold for $a(A,0)$, since $\mathfrak a(A,0)({\bf 1},{\bf 1}) = 0$.
The shifted form $b(A, 0)$ is therefore the appropriate form to
which the ultracontractivity theorem can be applied.\\
\end{rem}

\begin{lem}\label{Neumann_projection}
For every time $t \geq 0$ we have $\mathrm{e}^{-tL(A,0)}P = P$
in $\mathcal{L} \left(\mathrm{L}^{2}(\Omega) \right)$.
\end{lem}

\begin{proof}
Let $t \geq 0$ be arbitrary but fixed. Using Lemma \ref{1_Neumann_eigenvalue} we see that ${\bf 1}$ is an eigenfunction
of $L(A,0)$ to the eigenvalue $0$. So 
$\mathrm{e}^{-tL(A,0)}{\bf 1} = {\bf 1}$
holds almost everywhere in $\Omega$.
Hence $\mathrm{e}^{-tL(A,0)}Pu = Pu$ for every $u \in \mathrm{L}^{2}(\Omega)$ is implied by
$$
Pu \ = \ \frac{1}{|\Omega|} \langle u, {\bf 1} \rangle_{\mathrm{L}^{2}(\Omega)} {\bf 1}.
$$
\end{proof}

Let us now prove the main result of this subsection.\\

\begin{thm}\label{lower_bound_phi}
There exists a constant $\delta > 0$ such that 
$\delta \leq \phi$ holds pointwise almost everywhere in $\Omega$.
\end{thm}

\begin{proof}
\begin{enumerate}[i)]
\item Let $s_{0} > 0$ be arbitrary but fixed. Using Lemma \ref{Neumann_projection} we conclude
$$
\mathrm{e}^{-tL(A,0)} - P = \mathrm{e}^{-s_{0}L(A,0)} \left( \mathrm{e}^{-(t-s_{0})L(A,0)} - P \right)
$$
for any $t \geq s_{0}$. Using Lemma \ref{Neumann_ultracontractivity} we argue that
\begin{align*}
P\phi - \mathrm{e}^{-tL(A,0)}\phi & \leq \left| \mathrm{e}^{-tL(A,0)} \phi - P \phi \right| \\
& \leq \left\|  \mathrm{e}^{-s_{0}L(A,0)} \right\|_{\mathrm{L}^{2} \to \mathrm{L}^{\infty}} 
\left\|  \mathrm{e}^{-(t-s_{0})L(A,0)} - P \right\|_{\mathrm{L}^{2} \to \mathrm{L}^{2}} \| \phi \|_{\mathrm{L}^{2}(\Omega)}
\end{align*}
holds almost everywhere in $\Omega$ for every $t \geq s_{0}$.\\

\item Notice that $P\phi = \frac{1}{|\Omega|} \left( \int_{\Omega} \phi \d x \right) {\bf 1}$ is a constant function in 
$\mathrm{L}^{2}(\Omega)$. We define
$$
m_{\phi} =  \frac{1}{|\Omega|} \left( \int_{\Omega} \phi \d x \right) > 0
$$
such that $P \phi = m_{\phi} {\bf 1}$.

\item We choose $t_{0} = t_{0}(\phi) > s_{0}$ sufficiently large that
$$
\left\|  \mathrm{e}^{-(t_{0} - s_{0})L(A,0)} - P \right\|_{\mathrm{L}^{2} \to \mathrm{L}^{2}} <
\left\|  \mathrm{e}^{-s_{0}L(A,0)} \right\|_{\mathrm{L}^{2} \to \mathrm{L}^{\infty}}^{-1}\| \phi \|_{\mathrm{L}^{2}(\Omega)}^{-1}
\frac{m_{\phi} }{2}
$$
holds. Furthermore, Lemma \ref{Neumann_asymptotic} justifies the choice of $t_{0}$. Thus i) implies
$$
\frac{m_{\phi} }{2} {\bf 1} \ \leq \ \mathrm{e}^{-t_{0}L(A,0)}\phi \ \leq \ \mathrm{e}^{-t_{0}L(A,-C)}\phi = \mathrm{e}^{-t_{0}\mu}\phi
$$
almost everywhere in $\Omega$ where we used $\phi \geq 0$, the positivity of $\mathrm{e}^{-t_{0}L(A,0)}$ and 
(\ref{Neumann_comparison_inequ}).\\

\item We define 
$$
\delta = \mathrm{e}^{t_{0}\mu} m_{\phi}/2.
$$ 
Since $m_{\phi} > 0$, we have $\delta>0$. Moreover, by Theorem \ref{mu_eigenvalue}, $\mu \leq 0$ implies 
$\delta \leq m_{\phi}/2$.
Finally, $\phi \geq \delta$ is true almost everywhere in $\Omega$.
\end{enumerate}
\end{proof}




\section*{Declaration of generative AI and AI-assisted technologies in the manuscript preparation process}

During the preparation of this work, the author used ChatGPT, provided
by OpenAI, as a supporting tool in the development and verification of
parts of the arguments concerning the two-sided bounds for the
eigenfunction, in particular the power-truncation argument used for the
upper bound. The tool was also used to improve the language and
readability of the manuscript. The author carefully reviewed and
edited all resulting material and takes full responsibility for the
mathematical arguments and the content of the published article.\\

\bibliography{references}
\bibliographystyle{plain}

\end{document}